\documentclass[hidelinks, 12pt]{amsart}

\usepackage{amsmath}
\usepackage[psamsfonts]{amssymb}
\usepackage{amsfonts}
\usepackage{algorithm2e}
\usepackage{graphicx}
\usepackage{verbatim}
\usepackage{dsfont}
\usepackage{tikz-cd}
\usepackage{mathrsfs}
\usepackage{color}
\usepackage[english]{babel}
\usepackage{fontenc}
\usepackage{indentfirst}
\usepackage{tikz}
\usetikzlibrary{matrix,arrows,decorations.pathmorphing}
\usepackage{algorithm2e}
\usepackage[all]{xy}
\usepackage[T1]{fontenc}
\usepackage{hyperref}
\usepackage{mathtools}
\usepackage{multicol}
\usepackage{enumerate}

\newtheorem{theorem}{Theorem}[section]
\newtheorem{prop}[theorem]{Proposition}
\newtheorem{lemma}[theorem]{Lemma}
\newtheorem{cor}[theorem]{Corollary}

\newtheorem{ex}[theorem]{Example}
\newtheorem{dfn}[theorem]{Definition}
\newtheorem{remark}[theorem]{Remark}

\newcommand{\abs}[1]{\left\lvert#1\right\rvert}
\newcommand{\norm}[1]{\left\lVert#1\right\rVert}
\newcommand{\op}[1]{\operatorname{#1}}
\renewcommand{\hat}[1]{\widehat{#1}}
\numberwithin{equation}{section}
\DeclareMathOperator{\im}{im}

\DeclareMathOperator{\R}{\mathbb{R}}

\def\barmu{\bar{\mu}}

\def\R{\mathbb{R}}
\def\Z{\mathbb{Z}}
\def\N{\mathbb{N}}

\def\cal R{\mathcal R}

\def\k{{\bf k}}

\author{Danijel Djordjevi\'c}
\email{danijel$\_$djordjevic@matf.bg.ac.rs}
\address{Faculty of mathematics, University of Belgrade, Studentski trg 16, 11158 Belgrade,
Serbia}

\author{Igor Uljarevi\'c}
\email{igoru@matf.bg.ac.rs}
\address{Faculty of mathematics, University of Belgrade, Studentski trg 16, 11158 Belgrade,
Serbia}

\author{Jun Zhang}
\email{jzhang4518@ustc.edu.cn}
\address{The Institute of Geometry and Physics, University of Science and Technology of China, 96 Jinzhai Road, Hefei Anhui, 230026, China}
 
\title{Quantitative contact Hamiltonian dynamics}

\begin{document}

\maketitle

\begin{abstract} This paper presents a systematic quantitative study of contact rigidity phenomena based on the contact Hamiltonian Floer theory established in \cite{MU}. Our quantitative approach applies to arbitrary admissible contact Hamiltonian functions on the contact boundary $M = \partial W$ of a weakly\textsuperscript{+}-monotone symplectic manifold $W$. From a theoretical standpoint, we develop a comprehensive contact spectral invariant theory. As applications, the properties of these invariants enable us to establish several fundamental results: contact big fiber theorem, sufficient conditions for orderability, existence results of translated points. Furthermore, we uncover a non-traditional filtration structure on contact Hamiltonian Floer groups, which we formalize through the introduction of a novel type of persistence modules, called gapped modules, that are only parametrized by a partially ordered set. Among the various properties of contact spectral invariants, we highlight that the triangle inequality is derived through an innovative analysis of a pair-of-pants construction in the contact-geometric framework.\end{abstract}

\tableofcontents

\section{Introduction}

\subsection{Background} The dynamics of a contact manifold $(M,\xi)$ is completely governed by contact Hamiltonians. Namely, given a contact isotopy $\varphi_t:M\to M$, there exists a unique contact Hamiltonian that generates it and vice versa. The correspondence between contact isotopies and contact Hamiltonians, however, is not canonical: only after fixing a contact form $\alpha$ on $M$, one can canonically associate a contact Hamiltonian $h_t:M\to\R$ to a contact isotopy. Here are the equations that determine the relation between a contact Hamiltonian $h_t$ and the vector field $X_{h_t}$ of its contact isotopy $\varphi^t_h:M\to M:$
\begin{equation}\label{eq:contact-vector-field} 
\alpha(X_{h_t})=-h_t \quad\text{and}\quad d\alpha(X_{h_t},\cdot)= dh_t - dh_t(R_\alpha)\alpha.
\end{equation}
Here, $R_\alpha$ stands for the Reeb vector field with respect to $\alpha$.

In the case where $M$ admits a strong filling by a weakly\textsuperscript{+}  monotone symplectic manifold $W$ (see Definition~\ref{def:weakly+monotone} on page~\pageref{def:weakly+monotone}), \emph{the contact Hamiltonian Floer homology} ${\rm HF}_\ast(W, h; \k)$ is well defined\footnote{ If the filling $W$ is not exact, $\k$ is taken to be an appropriate Novikov field.}. Throughout the paper, we assume that the coefficient ring $\k$ is a field.  The filling $W$ and the coefficient ring $\k$ are often suppressed  from the notation. Inspired by the classical Viterbo's construction of symplectic homology in \cite{Viterbo-spec} (based on Hamiltonian Floer homology groups ${\rm HF}_*(H^a)$ on a non-compact symplectic manifold where the Hamiltonian functions $H^a$ are linear with admissible slopes $a \in \R$), the contact Hamiltonian Floer homology ${\rm HF}_\ast(h)$ was first introduced in \cite{MU} as a result of establishing a general maximum principle (see Theorem 1.1 in \cite{MU}).

The present work develops a comprehensive framework of quantitative invariants through contact Hamiltonian Floer homology and establishes their applications to contact rigidity phenomena. This development constitutes a natural and systematic contact-geometric analogue of quantitative methods in symplectic geometry. On the contact geometry side, prior approaches relied principally on either symplectic homology ${\rm SH}_*(W)$ (which captures only the Reeb dynamics) or the generating function techniques (which typically restrict to 1-jet spaces). Our ${\rm HF}_\ast(W, h; \k)$ overcomes these limitations by enabling the study of arbitrary contact Hamiltonian dynamics across a considerably broad class of contact manifolds.

Given the substantial technical complexity underlying our quantitative methods, we have organized the exposition of this work as follows: Section \ref{sec-main-app} presents the principal applications of our framework to contact rigidity phenomena; the core theoretical developments --- including all necessary preliminaries --- are consolidated in Section \ref{sec-main-theo}, which comprises four key components: Section \ref{sec:pre} recalls the basic geometric setting in this paper based on strong fillings of contact manifolds as well as the construction of the contact Hamiltonian Floer homology; Section~\ref{sec:persistence} constructs a novel persistence module formalism for contact Hamiltonian systems; Section~\ref{sec:spectral} develops contact spectral invariants (valued in $\R$) derived from ${\rm HF}_\ast(W, h; \k)$ and analyzes their properties; Section~\ref{sec:quasi} establishes the theory of partial contact quasi-states and their associated quasi-measures. 

Note that each subsection in Section \ref{sec-main-theo} contains results of independent interest: Theorem \ref{thm-cont-persistence} in Section~\ref{sec:persistence} (its related general theory is developed in Section \ref{sec-gap-mod}); Theorem \ref{thm:spectral-properties} and Theorem \ref{thm:product} in Section~\ref{sec:spectral} (which are proved in Section \ref{sec-spectral-prop} and Section \ref{sec-pp}, respectively); Theorem \ref{thm-qs} and Theorem \ref{cor:qm} in Section~\ref{sec:quasi} (which are proved in Section \ref{sec:quasi-states} and Section \ref{sec:quasi-measures}, respectively). 

\medskip

\subsection{Applications} \label{sec-main-app} As the first application, we give an alternative proof of the contact big fibre theorem from \cite{SUV25, UZ25}. 

\begin{theorem} [Corollary 1.10 in \cite{SUV25}] \label{thm-bft} Let $M$ be a closed contact manifold that is fillable by a Liouville domain $W$ with non-zero symplectic homology. Then, any contact involutive map $F: M \to \R^N$ has a fibre which is not contact displaceable. \end{theorem}

Recall that a subset $A$ of a contact manifold $M$ is contact displaceable if there exists a contact isotopy $\varphi_t:M\to M$ such that $\varphi_0={\rm id}$ and $\varphi_1(A)\cap A=\varnothing.$ 
As in \cite{SUV25}, ``contact involutive'' in Theorem~\ref{thm-bft} means that there exists a contact form $\alpha$ on $M$ such that   $F$ is Reeb invariant with respect to $\alpha$ and such that the coordinate functions of $F,$ seen as contact Hamiltonains (with respect to $\alpha$), generate commuting contact isotopies. Notice that Reeb invariance cannot be removed from the definition of contact involutive map so that Theorem~\ref{thm-bft} still holds. This can be seen by considering an example of a circle filled by a higher-genus surface.

\begin{remark}
For a generic contact form, recent work of Oh \cite{oh2025foliation,oh2025strict} proves that the only Reeb invariant functions are the constant ones. Therefore, Theorem~\ref{thm-bft} is trivial for a generic contact form. On the other hand, any coorientable contact manifold of dimension greater than 1 admits contact forms, not generic in Oh's sense, with much more Reeb invariant functions. This can be seen, for instance, by considering transverse knots and their standard neighbourhoods. Thus, Theorem~\ref{thm-bft} provides a contact non-displaceable proper closed subset for every contact manifold that is fillable by a Liouville domain with non-vanishing symplectic homology. In addition, there are examples (see \cite[Section~2.1]{SUV25}) where Theorem~\ref{thm-bft} proves contact rigidity that cannot be detected by methods of smooth topology.  
\end{remark}

\begin{remark}
By applying Theorem~\ref{thm-bft} to a single Reeb invariant contact Hamiltonian, seen as a contact involutive map, one recovers the contact non-squeezing from \cite{U-selective}, in particular Theorems~1.2 and 1.3 from \cite{U-selective}.
\end{remark}

In fact, our methods work in a significantly more general framework than that of Liouville domains (as in Theorem \ref{thm-bft}). Namely, we can prove the following statement. 

\begin{theorem}\label{thm:bft2}
Let $M$ be a closed contact manifold that has a strong weakly\textsuperscript{+} monotone filling  $W$ such that the unit $e\in {\rm SH}_\ast(W)$ is not eternal. Then, any contact involutive map $F: M \to \R^N$ has a fibre which is not contact displaceable.
\end{theorem}

The weak\textsuperscript{+} monotonicity condition ensures that the symplectic homology ${\rm SH}_\ast(W)$ is well defined. We recall this condition in Definition~\ref{def:weakly+monotone} on page \pageref{def:weakly+monotone}. As in \cite{cant2024remarks}, we call a class $\theta\in{\rm SH}_\ast(W)$ {\it eternal} if it is contained in the image of the canonical morphism ${\rm HF}_\ast(H^a)\to {\rm SH}_\ast(W)$ for each (admissible) slope $a\in\R$. 

\begin{cor}\label{cor:sphere}
If $W$ is a weakly\textsuperscript{+} monotone strong filling of the standard contact sphere $\mathbb{S}^{2n+1}$, for $n\geqslant 1,$ then the unit $e\in {\rm SH}_\ast(W)$ is eternal. 
\end{cor} 
\begin{proof}
Assume the contrary. Then, Theorem~\ref{thm:bft2} implies that every contact involutive map on $\mathbb{S}^{2n+1}$ has a contact non-displaceable fibre. Since, there are Reeb invariant non-constant functions on $\mathbb{S}^{2n+1},$ for $n\geqslant 1,$ there exists a closed proper subset of the standard contact $\mathbb{S}^{2n+1}$ that is contact non-displaceable. This is a contradiction and the proof is finished. 
\end{proof}

It is a well known fact, originally proven by Smith (the proof can be found in Seidel's article \cite[Corollary~6.5]{seidel2008biased}), that the standard contact sphere $(\mathbb{S}^{2n+1}, \xi_{\rm std})$, for $n\geqslant 1,$ does {\it not} admit Liouville fillings with non-vanishing symplectic homology. Along these lines, Corollary~\ref{cor:sphere} can be seen as an analogue of this fact for non-exact fillings. As shown by Ritter in \cite[Theorem~5]{ritter2014floer}, the sphere does admit non-exact strong fillings with non-zero symplectic homology (e.g. $\mathcal{O}(-1)$). 

\medskip

As the second application, using our contact spectral invariants, we can in fact prove a more general statement than Corollary~\ref{cor:sphere}. This statement is expressed in terms of the orderability of the boundary contact manifold.

\begin{theorem}\label{thm:order}
Let $W$ be a weakly\textsuperscript{+} monotone strong filling of a closed contact manifold $M$. If the unit $e\in {\rm SH}_\ast(W)$ is not eternal, then $M$ is orderable. 
\end{theorem}
The proof of Theorem~\ref{thm:order} is given on page~\pageref{proof:order}. As a special case of Theorem~\ref{thm:order}, the boundary of a Liouville domain with non-vanishing symplectic homology is orderable. This special case was first proved in \cite{CCD-R}, strengthening a previous result from \cite{AM}.

\medskip

For the final application, recall that starting from work of Sandon \cite{sandon2012iterated,sandon2013morse}, the questions on the existence of translated points have gained considerable attention in contact geometry. Contrary to the initial opinion \cite[Conjecture~1.2]{sandon2013morse}, there are contactomorphisms $\varphi\in {\rm Cont}_0(M)$ of a closed contact manifold $M$ with no translated points. The first such example was constructed by Cant in \cite{Cant-no-tp}. On the other hand, Oh proved in \cite{oh2022contact} that $\varphi$ does have a translated point provided its oscillation energy
\[ \abs{\varphi}^{\rm osc}_\alpha:= \inf _{h, \varphi^1_h=\varphi} \int_0^1(\max h_t-\min h_t) dt \]
is sufficiently small (smaller than the minimal period $\rho(\alpha)$ of a closed Reeb orbit)\footnote{Both $\abs{\cdot}^{\rm osc}_\alpha$ and $\rho(\alpha)$ depend on the choice of a contact form $\alpha.$}. This was conjectured and proved in the Liouville fillable case by Shelukhin in \cite[Conjecture~28, Theorem~B]{Egor-cont}. We provide an alternative proof of Shelukhin's conjecture in the strongly fillable case.

\begin{theorem}\label{thm:she-conj}
Let $M$ be a closed contact manifold with a contact form $\alpha.$ Assume that $M$ admits a weakly\textsuperscript{+} monotone strong filling. Then, every $\varphi\in{\rm Cont}_0(M)$ with $\abs{\varphi}_\alpha^{\rm osc}<\rho(\alpha)$ has a translated point.
\end{theorem}

Similarly to the proof of Theorem \ref{thm:she-conj}, we also prove a restriction, expressed in terms of symplectic homology, for strong fillings of a closed contact manifold $M$ admitting a contactomorphism $\varphi\in{\rm Cont}_0(M)$ with no translated points.

\begin{theorem} \label{thm-tran-eternal}
\sloppy Let $M$ be a closed contact manifold such that there exists $\varphi\in{\rm Cont}_0(M) $ with no translated points. Let $W$ be a weakly\textsuperscript{+} monotone strong filling of $M$. Then, the unit $e\in{\rm SH}_\ast(W)$ is eternal.
\end{theorem}

\subsection{Novelty} During the writing of this paper, concurrent research efforts have been emerging in related topics. In this section, we delineate key distinguishing aspects to underscore the novelty of our work, while providing a brief comparison with existing approaches in the literature.

\medskip

(1) This paper, expanded from its earlier version \cite{djordjevic2023quantitative} (that dates back to year 2023), initiates quantitative studies of (closed string) contact Hamiltonian dynamics through the lens of a modern Floer theory --- contact Hamiltonian Floer homology ${\rm HF}_\ast(W, h; \k)$ (elaborated in Section \ref{sec:chfh}). Here, the term ``quantitative'' refers to our systematic approach of deriving numerical invariant from this Floer homology theory and apply it to obtain {\bf new rigidity phenomena} in contact geometry. These contact rigidities include multiple central topics in contact geometry, ranging from contact non-squeezing or more generally contact non-displaceability (Theorem \ref{thm-bft}, Theorem \ref{thm:bft2}), orderability (Theorem \ref{thm:order}), strong fillings of a contact manifold (Theorem~\ref{cor:sphere}), and the existence of translated points (Theorems \ref{thm:she-conj} and  \ref{thm-tran-eternal}). 


(2) In contrast to the conventional action filtration used in most Floer theories in symplectic geometry, this paper (more accurately, the earlier version \cite{djordjevic2023quantitative}) discovers a {\bf new filtration structure} on ${\rm HF}_\ast(W, h; \k)$ --- the time-shift from ``twisting by Reed dynamics''. This filtration is intrinsically linked to quantitative aspects of Rabinowitz Floer homology. Our filtration structure, though contemporaneous with Cant’s independent work \cite{Cant_shelukhin}, follows a distinct approach (see Remark \ref{c-DC} for a detailed comparison). Moreover, we formulate this filtration structure into a new type of persistence module language, called gapped module (see Section \ref{sec:persistence} and Section \ref{sec-gap-mod}). Notably, this persistence module framework introduces a partially ordered (but {\it not} totally ordered) parametrization set --- a feature that, to our knowledge, appears for the first time in a concrete geometric setting.


(3) A key conceptual advance in the proofs of Theorems \ref{thm-bft} and \ref{thm:bft2} is the construction of contact-geometric counterparts to Entov-Polterovich's {\bf quasi-state and quasi-measure} (Definitions \ref{dfn-pqs} and \ref{dfn-qm}). To our knowledge, this is the first systematic definition of such invariants in such a general contact setting. The only prior related work we are aware of is \cite{BZ-GT}, which restricts its scope to prequantizations of certain symplectic toric domains. While Theorem \ref{thm:bft2} follows directly from the formal axioms of contact quasi-measures, the existence (established in Theorem \ref{thm-qs}, Corollary \ref{cor:qs}, and Corollary \ref{cor:qm}) is fundamental. Following Chapter 5 in \cite{PR14} and recent work \cite{dickstein2024symplectic}, these new tools provide a powerful foundation for addressing refined quantitative problems in contact geometry - particularly those with physical motivation and applications.


(4) This paper (again, back to its earlier version \cite{djordjevic2023quantitative}), for the first time, establishes and proves the {\bf triangle inequality of contact spectral invariant} (see the item 5 in Theorem \ref{thm:spectral-properties}). This key result is achieved through a careful analysis of the contact-geometric version of the pair-of-pants product (developed in Section \ref{sec-pp}). Subsequently, Cant obtained an analogous result in Section 1.4.1 of \cite{cant2024remarks}. Our triangle inequality resolves a question left open in earlier work by Oh-Yu \cite{Oh-Yu}, where a version of contact spectral invariants was defined for 1-jet spaces. In particular, Question 1.17 of \cite{Oh-Yu} explicitly raised the triangle inequality as an unresolved problem (see also Oh's recent work \cite{Oh-Leg-Fukaya} for related discussions on more general pants products in this context).


(5) A growing body of evidence supports a fundamental dichotomy in the contact rigidity of a contact manifold $(M, \xi)$, determined by the vanishing or non-vanishing of symplectic homology for its Liouville filling $W$. If ${\rm SH}_*(W) = 0$, then notable flexibility includes the non-orderability of the standard contact sphere $(\mathbb S^n, \xi_{\rm std})$ and the resulting contact squeezing phenomenon, first discovered by Eliashberg-Kim-Polterovich in \cite{EKP06}. If ${\rm SH}_*(W) \neq 0$, then the contact manifold $(M, \xi)$ typically displays stronger rigidity properties (see Theorem \ref{thm-bft} for a representative result). Building on the proofs of Corollary \ref{cor:sphere} and Theorem \ref{thm:order}, this paper proposes that for weakly\textsuperscript{+} monotone strong fillings, the appropriate generalization of this algebraic criterion is whether the {\bf unit $e \in {\rm SH}_*(W)$ is eternal or not}. This perspective unifies and extends (see Proposition \ref{prop:uniteternal}) the existing dichotomy, offering a refined framework for understanding contact rigidity in a broader geometric setting.

\begin{remark} Related to (5) above, we point out that the contact (non)-rigidity of the boundary contact manifold $(\partial W, \lambda|_{\partial W})$ coincides with the symplectic (non)-rigidity of the Liouville filling $(W, d\lambda)$ itself. This coincidence is exemplified by recent results from \cite{Sun24,FZ25}, which establish that the non-vanishing of symplectic cohomology ${\rm SH}^*(W)$ is equivalent to the heaviness (in the Entov-Polterovich sense \cite{EP09}) of the skeleton of $W$ within $(W, d\lambda)$. Intriguingly, this topological rigidity has profound implications for (the {\rm non-rigidity} of) Hamiltonian dynamics: when ${\rm SH}^*(W) \neq 0$, the group of compactly supported Hamiltonian diffeomorphisms $({\rm Ham}_c(W, \omega), d_{\gamma})$, endowed with the spectral norm denoted by $d_{\gamma}$, contains a quasi-isometric embedding of the infinite-dimensional function space $(C^{\infty}([0,1]; \R); d_{C^0})$. This suggests remarkably complicated Hamiltonian dynamics in $(W, d\lambda)$.  Conversely, as shown in \cite{BK22}, the vanishing of ${\rm SH}^*(W)$ implies a certain form of symplectic flexibility that constrains $({\rm Ham}_c(W, \omega), d_{\gamma})$ to be metrically bounded, demonstrating a sharp contrast in dynamical behavior between these two cases. \end{remark}

\section{Main theoretical results} \label{sec-main-theo} 

\subsection{Preliminaries} \label{sec:pre} In this section, we will recall several notations and concepts as preparations to the proofs of main results. 

\subsubsection{Strong fillings} In this paper, the completions of strong fillings play the role of the ambient symplectic manifold where the Floer theory actually happens. Therefore, let us recall its definition first. 

\begin{dfn}\label{def:strong-filling}
A \emph{strong filling} of a contact manifold $(M,\xi)$ is a compact symplectic manifold $(W,\omega)$ together with a 1-form $\lambda$ defined in a collar neighbourhood of the boundary $\partial W$ such that
\begin{enumerate}
\item $d\lambda=\omega$ in the region of $W$ where $\lambda$ is well defined;
\item the vector field $X_\lambda$, defined by $\omega(X_\lambda, \cdot)=\lambda,$ points transversely outwards at the boundary $\partial W;$
\item the contact manifolds $(M,\xi)$ and $(\partial W, \ker \left.\lambda\right|_{\partial W})$ are contactomorphic.
\end{enumerate}
\end{dfn}

The (locally defined) 1-form $\lambda$ in Definition~\ref{def:strong-filling} is a part of the data for the strong filling. Since the restriction $\alpha:=\left.\lambda\right|_{\partial W}$ is a contact form on $\partial W$, a strongly filled contact manifold comes with a preferred contact form.

For every strong filling $(W,\omega,\lambda)$, there exists a collar neighbourhood of the boundary that is canonically isomorphic to a part $(\partial W)\times(\varepsilon, 1]$ of the symplectization of $\partial W$. By attaching to $W$ the rest of the symplectization, we obtain the \emph{completion} $\widehat{W}$ of $W$. In the completion $\widehat{W}$, one can distinguish the conical end $\partial W\times(\varepsilon, +\infty)$ where the symplectic form is exact. A special case of strong fillings are \emph{Liouville domains}. The Liouville domains are those strong fillings for which the Liouville form $\lambda$ is defined not only in a collar neighbourhood of $\partial W$ but on the whole symplectic manifold $W.$

We are particularly interested in the strong fillings for which the symplectic homology theory is well defined. A general family of such strong fillings is provided by those that are \emph{weakly\textsuperscript{+} monotone}. The following definition recalls the notion of weak\textsuperscript{+} monotonicity \cite[Lemma~1.1]{HS1995}, \cite[Assumption ($\rm W^+$)]{seidel1997pi}, \cite[Section~2B]{ritter2016circle}.

\begin{dfn}\label{def:weakly+monotone}
A $2n$-dimensional symplectic manifold $(W,\omega)$ is \emph{weakly\textsuperscript{+} monotone} if at least one of the following conditions is satisfied:
\begin{enumerate}
\item $\left.\omega\right|_{\pi_2(W)}=0;$
\item $\left.c_1\right|_{\pi_2(W)}=0;$
\item there exists $s>0$ such that  $\left.(\omega-sc_1)\right|_{\pi_2(W)}=0;$
\item the minimal Chern number of $W$ is at least $n-1.$
\end{enumerate}
\end{dfn}
Here, $c_1$ denotes the first Chern class of the tangent bundle $TW$, which belongs to the cohomology group $H^2(W; \R)$. The notation $c_1|_{\pi_2(W)}$ means the evaluation of the cohomology class $c_1$ on the image of $\pi_2(W)$ in $H_2(W; \Z)$ under the Hurewicz map. 

\subsubsection{Contact Hamiltonian Floer homology}\label{sec:chfh} In this section, we recall the definition of the contact Hamiltonian Floer homology. The contact Hamiltonian Floer homology is associated to a weakly\textsuperscript{+} monotone strong filling $W$ and an \emph{admissible} contact Hamiltonian $h:[0,1]\times\partial W\times \R_{>0} \to\R$ on its boundary. The following definition explains what ``admissible'' means.

\begin{dfn}\label{def:admissible}
Let $M$ be a contact manifold with a contact form $\alpha$. A contact Hamiltonian $h:[0,1]\times M\to\R$ is called \emph{admissible} if the (symplectic) Hamiltonian 
\[ H: [0,1]\times M\times \R_{>0}\to \R\quad:\quad (t,x,r)\mapsto h_t(x)\cdot r \]
on the symplectization $M\times \R_{>0}$ of $M$ has no 1-periodic orbits and if $h_0=h_1$.
\end{dfn}
In the case where $h$ is constant, the admissibility condition simply means that $h$ is not a period of a Reeb orbit. Now, let $H:[0,1]\times\widehat{W}\to\R$ be a Hamiltonian on the completion $\widehat{W}$ of a weakly\textsuperscript{+} monotone strong filling that satisfies
\[ H_t(x,r)= h_t(x)r+c \]
on the conical end $\partial W\times(\varepsilon, +\infty)$ for $r$ large enough, for some admissible contact Hamiltonian $h:[0,1]\times\partial W\to\R,$ and for a constant $c\in\R.$ The contact Hamiltonian $h$ is called the \emph{slope} of the Hamiltonian $H$. By proving a maximum principle (based on Alexandrov's maximum principle), it was shown in \cite{MU} that the (symplectic) Hamiltonian Floer homology ${\rm HF}_\ast(H)$ is well defined. In addition, the continuation map ${\rm HF}_\ast(H)\to {\rm HF}_\ast(G)$ is well defined whenever $H$ and $G$ have admissible slopes $h$ and $g$ such that $h\leqslant g$ (defined pointwise).  As a consequence, the groups ${\rm HF}_\ast(H)$ and ${\rm HF}_\ast(G)$ are canonically isomorphic if $H$ and $G$ have the same slope. Thus, the Floer homology ${\rm HF}_\ast (h)$, called \emph{contact Hamiltonian Floer homology}, for an admissible contact Hamiltonian $h:[0,1]\times \partial W\to\R,$ is well defined. The continuation maps descend to contact Hamiltonian Floer homology. Explicitly, the continuation map ${\rm HF}_\ast (h)\to {\rm HF}_\ast(g)$ is defined for admissible contact Hamiltonians $h,g:[0,1]\times\partial W\to\R$ such that $h\leqslant g.$

The contact Hamiltonian Floer homology, with its continuation maps, can be seen as a directed system indexed by the set of all admissible contact Hamiltonians. Since the constant contact Hamiltonians form a cofinal subset in the index set, the direct limit of this directed system is  equal to the ``classical'' Viterbo symplectic homology
\[ \underset{h}{\lim_{\longrightarrow}}\: {\rm HF}_\ast (h)= {\rm SH}_\ast(W).\]

\subsubsection{Zig-zag isomorphisms} \label{sssec-zz} Now, we recall the construction of zig-zag isomorphisms from \cite[Section~9]{U-selective}. The zig-zag isomorphism
\[\mathcal B(\{h^s\}_{s \in [0,1]}): {\rm HF}_*(h^0) \to {\rm HF}_*(h^1) \]
is associated to a smooth $s$-family of \emph{admissible} contact Hamiltonians $h^s:[0,1]\times M\to\R, s\in[0,1].$ It is constructed as a composition of continuation maps and inverses of continuation maps as in the following diagram: 
\[ \xymatrixcolsep{1.5pc}\xymatrix{
{\rm HF}_*(h^0)  && {\rm HF}_*(h^{s_1})  && \,\,\,\,\,\cdots &&  {\rm HF}_*(h^1).\\
& {\rm HF}_*(g^0) \ar[lu]_-{\Phi^0} \ar[ru]^-{\Psi^1}  && \cdots \ar[lu]_-{\Phi^1} \,\,\, \ar[ru] && {\rm HF}_*(g^{k-1}) \ar[lu] \ar[ru]^-{\Psi^{k-1}}}
\]
More explicitly, the admissible contact Hamiltonians $h^{s_i}$ for $i = 0, ..., k$ are first chosen sufficiently dense along the family $h^s, s\in[0,1]$; one then chooses admissible contact Hamiltonians $g^j$ for $j=0,\ldots, k-1$ such that $g_j\leqslant \min\{ h^{s_j}, h^{s_{j+1}} \}$ and such that $g_j$ is sufficiently close to $h^{s_j}$ and $h^{s_{j+1}}$; finally, $\Phi^i$ and $\Psi^j$ are the isomorphisms given by the Floer continuation maps and 
$\mathcal B(\{h^s\}_{s \in [0,1]}) : = \Psi^{k-1} \circ \cdots \circ (\Phi^1)^{-1} \circ \Psi^1 \circ (\Phi^0)^{-1}.$ The following important observation is the reason why this construction is possible: the set of admissible contact Hamiltonians is an open subset, in $C^2$ sense, of the set of all the smooth functions on $[0,1] \times M$. The zig-zag isomorphisms preserve the grading and behave well with respect to the concatenation ``$\bullet$'' of smooth $s$-families of admissible contact Hamiltonians:
\begin{equation} \label{cc-formula}
\mathcal B(\{f^s\}_{s \in [0,1]} \bullet \{h^s\}_{s \in [0,1]}) = \mathcal B(\{h^s\}_{s \in [0,1]}) \circ \mathcal B(\{f^s\}_{s \in [0,1]}) 
\end{equation}
whenever $f^1 = h^0$. The concatenation ``$\bullet$'' is explicitly defined in (\ref{dfn-concatenation}).  In addition, the zig-zag isomorphisms behave well with respect to the continuation maps. Namely, if $\{h^s\}$ and $\{g^s\}, s\in[0,1]$ are two smooth $s$-families of admissible contact Hamiltonians such that $h^s\leqslant g^s$ for all $s\in[0,1]$ the following diagram, consisting of continuation maps and zig-zag isomorphisms, commutes: 
\[
\begin{tikzcd}
    {\rm HF}_\ast(h^0)\arrow{r}{\mathcal{B}(\{h^s\})}\arrow{d}& {\rm HF}_\ast(h^1)\arrow{d}\\
    {\rm HF}_\ast({g}^0) \arrow{r}{\mathcal{B}(\{{g}^s\})}& {\rm HF}_\ast({g}^1).
\end{tikzcd}
\] 
As a corollary, the following proposition holds.

\begin{prop}\label{prop:zigzagSH}
Let $W$ be a weakly\textsuperscript{+} monotone strong filling of a contact manifold $M$. Let $h^s:[0,1]\times M\to\R, s\in[0,1]$ be a smooth $s$-family of admissible contact Hamiltonians. Then, the following diagram, consisting of a zig-zag isomorphism and canonical morphisms,  commutes:
\[
\begin{tikzcd}
        &{\rm SH}_\ast(W)& \\
        {\rm HF}_\ast(h^0)\arrow{rr}{\mathcal{B}(\{h^s\})}\arrow{ru} & & {\rm HF}_\ast(h^1).\arrow{lu} 
    \end{tikzcd}
\] 
\end{prop}
\begin{proof}
Since $M$ is compact, there exists $a\in\R$ that is not a period of a Reeb orbit such that $h^s\leqslant a$ for all $s\in[0,1].$ The zig-zag isomorphism with respect to the constant family $g^s\equiv a$ is equal to the identity ${\rm HF}_\ast(a)\to {\rm HF}_\ast(a).$ Therefore, the following diagram commutes:
\[
\begin{tikzcd}
        &{\rm HF}_\ast(a)& \\
        {\rm HF}_\ast(h^0)\arrow{rr}{\mathcal{B}(\{h^s\})}\arrow{ru} & & {\rm HF}_\ast(h^1).\arrow{lu} 
    \end{tikzcd}
\] 
Composing with the canonical morphism ${\rm HF}_\ast(a)\to{\rm SH}_\ast(W)$ finishes the proof.
\end{proof}

\subsubsection{Non-eternal unit} Let us elaborate on an algebraic condition based on the (non)-eternal property of the unit in symplectic homology.
By Poincare duality (see Section~2.5 in \cite{cieliebak2010rabinowitz}) the inverse limit
$\displaystyle\lim_{\longleftarrow}{\rm HF}_\ast(a)$
can be identified with ${\rm SH}^{-\ast}(W)$ (where ${\rm SH}^*(W)$ denotes the symplectic cohomology of $W$). Therefore, there is a canonical morphism
\begin{equation}\label{eq:ast2ast}
{\rm SH}^{-\ast}(W)\to {\rm SH}_\ast(W),
\end{equation}
which can be conveniently used to express what an eternal class is. Namely, $\theta\in{\rm SH}_\ast(W)$ is eternal if, and only if, it is contained in the image of the map \eqref{eq:ast2ast}.

\begin{lemma}\label{prop:eternal}
The unit $e\in{\rm SH}_\ast(W)$ is eternal if, and only if, the elements of ${\rm SH}_\ast(W)$ are all eternal. In fact, if the unit $e$ is eternal, then the map \eqref{eq:ast2ast} is an isomorphism.  
\end{lemma}
\begin{proof}
Obviously, if the elements of ${\rm SH}_\ast(W)$ are all eternal, then in particular $e$ is eternal. Let us prove the opposite direction. Assume there exists $\theta\in{\rm SH}_\ast(W)$ that is not eternal. Then, the number
\begin{equation}\label{eq:a}
a:=\inf\{\eta\in\R\:|\: \theta\in {\rm im} ({\rm HF}_\ast(\eta)\to {\rm SH}_\ast(W))\}
\end{equation}
is finite. Let $\varepsilon>0.$ Since $e$ is eternal, there exists $\theta_0\in{\rm HF}_\ast(-2\varepsilon)$ that is mapped to $e$ via the canonical map ${\rm HF}_\ast(-2\varepsilon)\to {\rm SH}_\ast(W).$ Similarly, since $a+\varepsilon>a,$ there exists $\theta_1\in {\rm HF}_\ast(a+\varepsilon)$ that is mapped to $\theta$ via the canonical map ${\rm HF}_\ast(a+\varepsilon)\to {\rm SH}_\ast(W).$ The existence of the product (see Theorem \ref{thm:product} or Section \ref{sec:defPPP})
\[ {\rm HF}_\ast(-2\varepsilon)\otimes {\rm HF}_\ast(a+\varepsilon)\to {\rm HF}_\ast(a-\varepsilon), \]
that behaves well with respect to the canonical morphisms, implies that $\theta_0\ast\theta_1$ is mapped to $\theta$ via the canonical map ${\rm HF}_\ast(a-\varepsilon)\to {\rm SH}_\ast(W).$ This contradicts the definition of $a$ because $a>a-\varepsilon$ and proves the first part of the proposition. The rest of the proposition, that is the injectivity of \eqref{eq:ast2ast}, follows from the same argument after applying the Poincare duality.
\end{proof}

The following proposition will be repeatedly applied in later proofs. 

\begin{prop}\label{prop:uniteternal}
If $W$ is a Liouville domain with ${\rm SH}_\ast(W)\not=0,$ then the unit $e\in {\rm SH}_\ast(W)$ is not eternal.
\end{prop}
\begin{proof}
Assume the contrary, i.e. ${\rm SH}_\ast(W)\not=0$ and $e$ is eternal. By Lemma~\ref{prop:eternal}, the morphism ${\rm SH}^{-\ast}(W)\to {\rm SH}_\ast(W)$ is an isomorphism. By Theorems~1.2 and 1.5 in \cite{cieliebak2010rabinowitz}, this map fits into the long exact sequence
\[\cdots\to {\rm SH}^{-\ast}(W) \to {\rm SH}_\ast(W) \to {\rm RFH}_\ast(W) \to {\rm SH}^{-\ast +1}(W)\to\cdots. \]
Hence, the Rabinowitz Floer homology ${\rm RFH}_\ast(W)$ vanishes. By \cite[Theorem~13.3]{ritter2013topological}, ${\rm RFH}_\ast(W)$ vanishes if and only if ${\rm SH}_\ast(W)$ vanishes. This is a contradiction that completes the proof.
\end{proof}

\subsection{Contact persistence module}\label{sec:persistence}

Following the analogy between the Morse and Floer theories, the (symplectic) Hamiltonian Floer theory can be upgraded to a persistence module \cite{CZ05,CC-SGGO-prox-09,CdeSGO16,PRSZ20} using the filtration by the action functional. This relatively simple algebraic structure enriches the classical homological theory by allowing homologically invisible generators, as well as their homological-killing relations with respect to the Floer boundary operator, to be studied systematically. Consequently, the persistence module perspective provides a uniform framework for extracting numerous homological invariants from Hamiltonian Floer homology, including spectral invariants (and spectral norm) \cite{Oh-spec,Viterbo-spec,Schwarz}, boundary depths \cite{Usher-bd-1,Usher-bd-2,UZ16}, etc. These results have made fundamental contributions to the quantitative studies in the modern symplectic geometry, starting from Polterovich-Shelukhin's pioneering work \cite{PS-pers-16}.

The action filtration does not directly descend to the contact Hamiltonian Floer homology ${ \rm HF}_\ast(h).$ For this reason, we propose a novel approach to obtain a persistence module structure for ${\rm HF}_\ast(h)$ that does not use the action filtration. Our approach, which we describe in the following lines, allows a more direct dynamical interpretation.
Throughout the paper, for contact Hamiltonians $h,g:[0,1]\times M\to\mathbb{R}$, we denote by $h\# g$ the contact Hamiltonian that generates the composition of contact isotopies $\varphi^t_h$ and $\varphi^t_g$. Explicitely, $h\#g$ is given by
\begin{equation}\label{eq:composition}
(g\# h)_t:= g_t + (\kappa^t_g h_t)\circ (\varphi_g^t)^{-1}.
\end{equation}
Here $\kappa_g^t: M\to\mathbb{R}_{>0}$ stands for the conformal factor of the contactomorphism $\varphi^t_g$ with respect to a fixed contact form $\alpha,$ in other words, $\kappa^t_g$ is the smooth function defined by $(\varphi_g^t)^\ast \alpha = \kappa_g^t\alpha.$ In particular, if $g\equiv \eta\in\mathbb{R}$ is a constant contact Hamiltonian, it generates the reversed\footnote{Here, we use the following sign convention: for a Liouville domain $(W, \lambda)$, its symplectic structure is given by $\omega = d\lambda$ and $\iota_{X_H} \omega = \omega(X_H, \cdot) = dH$. By this sign convention, a contact Hamiltonian vector field $X_h$ (generated by $h$) is uniquely determined by the equations $\alpha(X_h)  = -h$ and $\iota_{X_h}d\alpha = dh - dh(R_{\alpha})\alpha$ (as in \eqref{eq:contact-vector-field} on page \pageref{eq:contact-vector-field}).} reparametrized Reeb flow $\varphi_R^{-\eta t}$ and
\begin{equation}\label{comp} (\eta\# h)_t= \eta + h_t\circ \varphi_R^{\eta t}. \end{equation} 

Now, to a contact Hamiltonian $h:[0,1]\times M\to\R$ and a weakly\textsuperscript{+}  monotone strong filling $W$ of $M$, one can associate the persistence module
\begin{equation} \label{per-mod-CHD}
\mathbb{P}(W,h):=\left\{ {\rm HF}_\ast(W,\eta\#h) \right\}_\eta.
\end{equation}
The morphisms of the persistence module $\mathbb{P}(W,h)$ are the standard continuation maps. We write $\mathbb{P}(h)$ instead of $\mathbb{P}(W,h)$ when there is no danger of confusion. In general, the contact Hamiltonian Floer homology ${\rm HF}_\ast(f)$ is well defined when $f$ is admissible, that is when the Hamiltonian
\[M\times\R_{>0}\to \R\quad:\quad (x,r)\mapsto rf_t(x)\]
on the symplectization of $M$ has no 1-periodic orbits (see Definition~\ref{def:admissible} on page~\pageref{def:admissible} for more details). Consider $f$ in the form of $\eta \# h$, then the persistence module $\mathbb{P}(h)$ is indexed by a subset of $\R$ consisting of those numbers $\eta$ for which $\eta\# h$ is admissible. Geometrically, this set is the complement of the time-shifts of the translated points of $\varphi_h^1$. Recall that $x\in M$ is a translated point of $\varphi_h^1$ if there exists $\eta\in\R$ such that $\kappa_h^1(x)=1$ and $\varphi_h^1(x)= \varphi_R^\eta(x)$. The number $\eta$ is called a \emph{time-shift} of the translated point $x$. Hence, $\mathbb{P}(h)$ is indexed by $\R\setminus\mathcal{S}_h$ where 
\begin{equation}\label{eq:spectrum}\mathcal{S}_h:=\left\{\eta\in\R\:|\: \eta \text{ is a time-shift of a translated point of }\varphi_h^1\right\}.\end{equation}
The set $\mathcal{S}_h$ is a closed subset of $\R$. In addition, up to a sign, $\mathcal{S}_h$ coincides with the spectrum of the Rabinowitz-Floer twisted action functional. Hence, $\mathcal{S}_h$ is a nowhere dense subset of $\R$ (see \cite[Remark~2.10]{AM} and \cite[Lemma~3.8]{Schwarz}).
One peculiarity of the persistence module $\mathbb{P}(h)$ is that it is not indexed by a totally ordered poset. In other words, $\mathbb{R}\setminus\mathcal{S}_h$ is not endowed with the standard order relation. In order to describe the order relation on $\mathbb{R}\setminus\mathcal{S}_h$ that we consider, denote
\[ {\rm osc}_R h:= \max_{t\in[0,1], \,x\in M} \left( \sup _{s\in\R}(h_t \circ\varphi_R^s(x))- \inf_{s\in \R} (h_t\circ\varphi_R^s(x))\right).\]
 Geometrically, ${\rm osc}_R h$ measures how much the contact Hamiltonian $h$ oscillates along the Reeb trajectories. Our persistence module $\mathbb{P}(h)$ is indexed by the poset $(\R\setminus\mathcal{S}_h,\prec),$ where the order relation $\prec$ is defined as follows,
\begin{equation} \label{po-CHD}
 \eta\prec \eta' \,\,\,\,\mbox{iff} \,\,\,\,\mbox{either $\eta=\eta'$ or $\eta'-\eta \geqslant {\rm osc}_R h$}.
 \end{equation} 

There is no major issue for defining spectral invariants of $\mathbb{P}(h)$ despite $\prec$ not being totally ordered, (see Section~\ref{sec:spectral} below and cf.~Definition \ref{dfn-si-gap} in Section \ref{ssec-alg-si}). On the other hand, one cannot associate to $\mathbb{P}(h)$ a barcode in a usual way, because at present only indecomposable persistence modules over totally ordered parameter sets admit an accurate description in terms of intervals (see \cite[Theorem 1.2]{BC-B20}). In Section~\ref{sec-gap-mod}, we develop a general theory for this type of persistence modules, called gapped modules (Definition \ref{dfn-gap-mod} in Section \ref{sec-gap-mod}). In particular, we overcome the challenge that the parameter set is not totally ordered by passing to discrete subsets $\{\ldots, \eta_{-1}, \eta_0, \eta_1,\ldots\}$ of $\R\setminus\mathcal{S}_h$ for which the gap $\eta_{j+1}-\eta_j$ between consecutive elements is more or less equal to ${\rm osc}_R h.$

\medskip

The discussion above summarizes into the following result. 

\begin{theorem} \label{thm-cont-persistence} Let $M$ be a contact manifold that admits a weakly\textsuperscript{+}  monotone strong filling $W$. Then, any contact Hamiltonian $h:[0,1]\times M\to\R$ gives rise to an ${\rm osc}_R h$-gapped module $\mathbb{P}(W,h)$ defined as in (\ref{per-mod-CHD}). \end{theorem}

Example \ref{ex-alg-csi}, together with Definition \ref{dfn-gap-inter}, elaborates on how two such gapped modules, induced by different contact Hamiltonian inputs, are quantitatively related. 

\begin{remark} \label{c-DC} While the first version \cite{djordjevic2023quantitative} of this paper form 2023 was close to completion, we were informed by Cant about his, at that time, ongoing project \cite{Cant_shelukhin} that also constructs a persistence module from the contact Hamiltonian Floer homology ${\rm HF}_*(W, h; \k)$ with a certain amount of similarity to our construction\footnote{The two preprints appeared independently on ArXiv on the same date.}.   More precisely, he also starts from the contact Hamiltonians in a similar form of (\ref{comp}). Different from our gapped module approach proposed in Section \ref{sec-gap-mod}, he is able to obtain a standard persistence module parametrized by $\R$ or a dense subset of $\R$, via twisted cylinders along a homotopy $\{h\#\eta_{s}\}_{s \in [0,1]}$ from $h\#\eta$ to $h\# \eta'$ for any two $\eta \leq \eta'$. For more details, see \cite[Section 2]{Cant_shelukhin}. \end{remark} 

\subsection{Contact spectral invariants}\label{sec:spectral}

In this section, we introduce contact spectral invariants $c(W, h, \theta)$ and discuss their properties. The contact spectral invariant $c(W, h, \theta)$ is associated to a weakly\textsuperscript{+} monotone strong filling $W$ of a contact manifold $M,$ to a contact Hamiltonian $h: [0,1]\times M\to \R$, and to an element $\theta\in {\rm SH}_\ast(W)$ of the symplectic homology of $W$. When there is no danger of confusion, we suppress $W$ form the notation and write $c(h,\theta)$ instead of $c(W, h,\theta).$ An important observation for the definition of $c(h,\theta)$ is that the contact Hamiltonians $ \eta\# h, \eta\in\R\setminus \mathcal{S}_h$ form a cofinal subset in the set of all admissible contact Hamiltonians. Therefore,
\[ \underset{\eta}{\lim_{\longrightarrow}}\: {\rm HF}_\ast (\eta\#h)= {\rm SH}_\ast(W).\]
The following definition introduces the spectral invariant $c(h,\theta).$

\begin{dfn}\label{def:spectral}
Let $h:[0,1]\times M\to\R$ be a contact Hamiltonian on a contact manifold $M$ that is strongly filled by a weakly\textsuperscript{+} monotone symplectic manifold $W$. Let $\theta\in{\rm SH}_\ast(W).$ Then, the spectral invariant $c(W, h, \theta)$ is defined by
\begin{equation}\label{eq:spectral} c(W, h, \theta):= -\inf\left\{ \eta\in\R\:\bigg|\; \theta\in {\rm im} \bigg({\rm HF}_\ast(\eta\#h) \to {\rm SH}_\ast(W)\bigg) \right\}.
\end{equation}
\end{dfn}

We are particularly interested in finite spectral invariants $c(h,\theta)$. Whether $c(h,\theta)$ is finite or infinite, depends entirely on $\theta$: if $c(h,\theta)=+\infty$ for some $h$, then $c(h,\theta)=+\infty$ for all $h$. The elements $\theta\in{\rm SH}_\ast(W)$ for which $c(h,\theta)=+\infty$ are precisely those that are contained in the image of the map  ${\rm HF}_\ast(\eta)\to{\rm SH}_\ast(W)$ for all $\eta\in\R.$ Such elements $\theta$ are called in \cite{cant2024remarks} \emph{eternal}. A trivial example of an eternal $\theta$ is $\theta=0$. Many non-trivial examples of eternal elements are provided by \cite{cant2024remarks}. 
The following example of a non-eternal element is particularly important in the present paper.

\begin{ex}\label{ex:LD-not-eternal}
If $W$ is a Liouville domain with ${\rm SH}_\ast(W)\not=0$, then the unit $e \in {\rm SH}_\ast(W)$ is \textbf{not} eternal (see Proposition~\ref{prop:uniteternal} on page~\pageref{prop:uniteternal}). In particular, $c(h,e)$ is finite for all smooth functions $h:[0,1]\times M\to\R$.
\end{ex}

The following theorem lists properties of the spectral invariants $c(h,\theta).$ Before stating it, we introduce some necessary notations. For two contact Hamiltonians $h,g:[0,1]\times M\to\R,$ denote by ${\rm osc}_R(h,g)$ the number
\[ {\rm osc}_R(h,g) := \max_{t\in[0,1], \,x\in M}\left( \sup_{s\in \R} \left(h_t\circ\varphi_R^s(x)\right) - \inf_{s\in \R} \left(g_t\circ\varphi_R^s(x)\right) \right). \]
This number represents the maximal difference between the maximum of $h$ and the minimum of $g$ along the same Reeb trajectory. Notice that ${\rm osc}_R(h,h)= {\rm osc}_R(h).$ Denote by $g\bullet h$ the concatenation of the contact Hamiltonians $g$ and $h$ defined by 
    \begin{equation} \label{dfn-concatenation}
     \left(g\bullet h\right)_t:= \left\{ \begin{matrix} 2g_{2t} & \text{for } t\in\left[0,\frac{1}{2}\right]\\ 2h_{2t-1}&\text{for } t\in\left[\frac{1}{2}, 1\right].\end{matrix} \right.
     \end{equation}
Note that $g \bullet h$ generates the contact isotopy 
\[ \varphi^t_{g\bullet h} = \left\{ \begin{matrix} \varphi_g^{2t} & \text{for } t\in\left[0,\frac{1}{2}\right]\\ \varphi_h^{2t-1} \varphi_g^1 &\text{for } t\in\left[\frac{1}{2}, 1\right].\end{matrix} \right.\]
The contact Hamiltonians $g\bullet h$ and $h\# g$ give rise to the same element in the universal cover $\widetilde{\rm Cont}_0(M)$ of the identity component of the contactomorphism group. By ``$\ast$'', we denote the pair-of-pants product of classes in symplectic homology.

\begin{theorem}\label{thm:spectral-properties}
Let $W$ be a weakly\textsuperscript{+} monotone strong filling of a contact manifold $M$, and let $\theta,\theta_1, \theta_2\in {\rm SH}_\ast(W)$ be elements that are {\rm not} eternal. Then, the following holds for all contact Hamiltonians $h,g:[0,1]\times M\to\R$:
\begin{enumerate}[(1)]
\item {\rm (Spectrality)} $c(h,\theta)\in -\mathcal{S}_h.$
\item {\rm (Shift)} $c(s\# h, \theta)= c(h,\theta) + s$ for all $s\in\R.$
\item {\rm (Monotonicity)} $c(h,\theta)\leqslant c(g, \theta)$ if $h\leqslant g.$
\item {\rm (Stability)} $ -{\rm osc}_R(g,h)\leqslant c(h,\theta) -c (g, \theta)\leqslant {\rm osc}_R(h,g). $
\item {\rm (Triangle inequality)} $c(h,\theta_1) + c(g, \theta_2)\leqslant  c(h\# g, \theta_1\ast \theta_2)+ 2\cdot \max\{{\rm osc}_R h, {\rm osc}_R g\}. $
\item {\rm (Descent)} $c(h, \theta)= c(g, \theta)$ whenever $h$ and $g$ give rise to the same class in $\widetilde{\rm Cont}_0(M).$
\end{enumerate}
\end{theorem}

While most of the statements in Theorem~\ref{thm:spectral-properties} follow more or less directly from the definition of $c(h,\theta)$, the triangle inequality required significant efforts. (Notice that, due to our sign conventions, the triangle inequality is ``reversed''.) Namely, the proof of the triangle inequality uses the existence of the product
\[ {\rm HF}_\ast(h)\otimes {\rm HF}_\ast(g)\to {\rm HF}_\ast(h\bullet g) \] 
for the contact Hamiltonian Floer homology. In Section~\ref{sec-pp}, we prove that such a product exists using a novel analysis based on Alexandrov's maximum principle.

\begin{theorem}\label{thm:product}
Let $W$ be a weakly\textsuperscript{+} monotone strong filling of a contact manifold $M$. Let $h,g:[0,1]\times M\to \R$ be admissible contact Hamiltonians such that $h_t=g_t=0$ for $t$ in a neighbourhood of $0$ or $1$ and such that $h\bullet g$ is also admissible. Then, there exists a product
\begin{equation}\label{eq:product}
\ast :{\rm HF}_\ast(h)\otimes {\rm HF}_\ast(g)\to {\rm HF}_\ast(h\bullet g)
\end{equation}
such that the following diagram commutes:
\[
\begin{tikzcd}
    {\rm HF}_\ast(h)\otimes {\rm HF}_\ast(g)\arrow{r}\arrow{d}& {\rm HF}_\ast(h\bullet g)\arrow{d}\\
    {\rm SH}_\ast(W)\otimes  {\rm SH}_\ast(W) \arrow{r}& {\rm SH}_\ast(W).
\end{tikzcd}
\] 
In the diagram above, the vertical arrows are induced by the canonical morphisms and the horizontal arrows are the products. \end{theorem}
\begin{proof}
This theorem is proven in Section~\ref{sec-pp}.
\end{proof}
 The present paper (actually, its earlier version \cite{djordjevic2023quantitative}) is the first to introduce the pair-of-pants product in contact Hamiltonian Floer homology.
 
\begin{remark}\label{rmk:spectral-strict}
In the case where the contact Hamiltonians $h, g$ from Theorem~\ref{thm:spectral-properties} are Reeb invariant, the stability property and the triangle inequality have the following more elegant forms: 
\begin{enumerate}[ ]
\item {\rm (Stability)} $\min (h-g)\leqslant c(h,\theta)-c(g, \theta)\leqslant \max (h-g),$
\item {\rm (Triangle inequality)} $c(h,\theta_1) + c(g, \theta_2)\leqslant c(h\# g, \theta_1\ast \theta_2).$
\end{enumerate}
\end{remark}

To end this section, let us mention a direct implication of Theorem \ref{thm:spectral-properties}. In \cite{Egor-cont}, Shelukhin introduced a Hofer-like norm, which we denote by $\abs{\cdot}_{\rm S},$ on the group $\widetilde{\rm Cont}_0(M)$. This norm is defined as follows
\[ \abs{\varphi}_{\rm S}:= \inf \int_0^1 \max_M\abs{h_t}dt, \]
where the infimum goes through all the contact Hamiltonians $h:[0,1]\times M\to\R$ such that $[\{\varphi_h^t\}_{t\in[0,1]}]=\varphi.$ We denote by $d_{\rm S}$ the corresponding distance function (defined by $d_{\rm S}(\varphi, \psi):= \abs{\psi\varphi^{-1}}_{\rm S}$). By property (6) in Theorem \ref{thm:spectral-properties}, for any $\varphi \in \widetilde{\rm Cont}_0(M)$, the notation $c(\varphi, \theta)$ is well defined for any class $\theta\in {\rm SH}_\ast(W)$. 

\begin{cor}\label{thm:shelukhin}
Let $W$ be a weakly\textsuperscript{+} monotone strong filling of a contact manifold $M$ and let $\theta\in {\rm SH}_\ast(W)$ be a class that is not eternal. Then,
\[ \abs{c(\varphi, \theta)- c([{\rm id}], \theta)}\leqslant \abs{\varphi}_{\rm S}\]
for all $\varphi\in \widetilde{\rm Cont}_0(M).$
\end{cor}
\begin{proof}
Let $h:[0,1]\times M\to\R$ be a contact Hamiltonian that induces $\varphi.$ Then, the stability property of the contact spectral invariants, i.e., property (4) in Theorem \ref{thm:spectral-properties}, implies
\[ \min h\leqslant c(h,\theta)- c([{\rm id}], \theta)\leqslant \max h.\]
Therefore, $\abs{c(h,\theta)- c([{\rm id}], \theta)}\leqslant \max\{\max h, \max (-h)\}= \max\abs{h}.$ The theorem follows from the following identity:
\[ \inf_h \int_0^1\max_{M} \abs{h_t} dt = \inf_h \max_{t\in[0,1], \,x\in M} \abs{h_t(x)}, \]
where the infima are taken over the set of all contact Hamiltonians $h:[0,1]\times M\to\R$ that give rise to $\varphi.$ The equality above can be easily verified via a reparametrization trick as in the proof of Lemma 5.1.C in \cite{Pol-green}.
\end{proof}

\begin{remark} A recent series of papers from Oh \cite{Oh-21,Oh-21-b,Oh-22, Oh-22-b, Oh-22-c} established a new approach to define a Floer homological theory on a contact manifold {\rm without} passing to any form of its symplectization. Compared with our approach, it is closer to the standard (symplectic) Hamiltonian Floer type theory where an appropriate action functional is constructed. Moreover, such an action functional serves as a Morse function on a proper loop space and the corresponding  critical values are precisely the time-shifts in our construction (see $\mathcal{S}_h$ above). This leads to the construction of some numerical invariants, e.g. contact spectral invariants, in certain cases {\rm (}see \cite{Oh-Yu-a,Oh-Yu}{\rm )}. \end{remark}

\subsection{Contact quasi-states and quasi-measures}\label{sec:quasi}

In this section, we introduce the notions of partial contact quasi-state and contact quasi-measure. This can be regarded as the contact-geometric development of Entov-Polterovich's quasi-state theory from \cite{EP06,EP08,EP09,Ent14}, or more broadly speaking Aarnes's quasi-state theory in \cite{Aar91} realized in contact geometry. We first recall the notion of a contact Poisson bracket. Given smooth functions $g, h: M \to \R$ on a contact manifold $M$ with a contact from $\alpha$, the contact Poisson bracket is defined by 
\[ \{g, h\}_{\alpha}: = dg(X_h) +g\cdot dh(R_{\alpha}) = - dh(X_g)- h\cdot dg(R_\alpha), \]
where $X_h$ is the contact Hamiltonian vector field of $h$. The expression for the contact Poisson bracket is obtained by passing to the symplectization and taking the ``slope'' of the (symplectic) Poisson bracket of the lifts of $h$ and $g$. Notice that $\{1, h\}_\alpha=0$ is equivalent to $h$ being Reeb invariant with respect to $\alpha$.

\begin{dfn}\label{dfn-pqs}  Let $M$ be a contact manifold with a contact form $\alpha$. A {\bf partial contact quasi-state} with respect to $\alpha$ is a functional $\zeta: C^\infty(M) \to  \R$ which satisfies the following axioms:
	\begin{enumerate}[(1)]
		\item \label{pqm:normalization} $\zeta(1)=1$.
		\item\label{pqm:homogeneous} For all $f\in C^\infty(M)$ and $s\in\R_{>0}$, we have $\zeta(sf)=s\zeta(f).$
		\item\label{pqm:stability} If $f, g\in C^\infty(M)$ are Reeb invariant with respect to $\alpha$, then
			\[\min_M(f-g)\leqslant \zeta(f) - \zeta(g) \leqslant \max_M (f-g).\]
		\item\label{pqm:conjugacy} If $\varphi\in{\rm Cont}_0(M)$ is Reeb invariant with respect to $\alpha$, then 
		\[\zeta(f\circ\varphi) = \zeta(f)\]
		for all $f\in C^\infty(M).$
		\item\label{pqm:vanishing} If $f$ is Reeb invariant with respect to $\alpha$ and ${\rm supp}\:f$ is displaceable by an element of ${\rm Cont}_0(M),$ then $\zeta(f)=0.$
		\item\label{pqm:triangle} The inequality
		\[  \zeta(f) + \zeta(g) \leqslant \zeta(f+g)\]
		holds for all $f,g\in C^\infty(M)$ such that $0 = \{g,f\}_\alpha = \{g, 1\}_\alpha = \{f, 1\}_\alpha.$ If, in addition, ${\rm supp}\:g$ is displaceable by an elemenot of ${\rm Cont}_0(M),$ then 
		\[\zeta(f+g)= \zeta(f) + \zeta(g) =\zeta(f). \]
	\end{enumerate}
\end{dfn}

In Section \ref{sec:quasi-states}, we use the contact spectral invariant $c(h,e)$, for the unit $e$ in the symplectic homology of the filling, to construct a partial contact quasi-state $\zeta_{\alpha}$.

\begin{theorem} \label{thm-qs} Let $(M,\xi)$ be a closed contact manifold that is strongly fillable by a weakly\textsuperscript{+} monotone symplectic manifold $W$ such that the unit in ${\rm SH}_\ast(W)$ is not eternal. Then, for every contact form $\alpha$ of $\xi$, there exists a contact quasi-state with respect to $\alpha$. \end{theorem}

Since the unit in ${\rm SH}_\ast(W)$ is not eternal if $W$ is a Liouville domain with non-zero symplectic homology, Theorem~\ref{thm-qs} applies to the boundaries $\partial W$ of such Liouville domains. Then we obtain the following result.

\begin{cor} \label{cor:qs}
Let $(M,\xi)$ be a closed contact manifold that is fillable by a Liouville domain $W$ such that ${\rm SH}_*(W)$ is not vanishing. Then, for every contact form $\alpha$ of $\xi$, there exists a contact quasi-state with respect to $\alpha$. 
\end{cor}

Before introducing the notion of contact quasi-measures we recall the definition of a contact involutive map. A smooth map $F=(f_1,\ldots, f_N):M\to\R^N$ on a contact manifold $M$ with a contact form $\alpha$ is called \emph{contact $\alpha$-involutive} if the Poisson brackets $\{f_j, 1\}_\alpha$ and $\{f_i,f_j\}_\alpha$ vanish for all $i,j=1,\ldots, N.$ For brevity, $F: M \to \R^N$ is called {\it contact involutive} if it is contact $\alpha$-involutive for some contact form $\alpha$. This notion agrees with Definition 1.7 in \cite{SUV25}.

\begin{dfn} \label{dfn-qm} Let $M$ be a closed contact manifold with a contact form $\alpha$. A  {\bf contact quasi-measure} with respect to $\alpha$ is a map $\tau: {\rm Cl}(M) \to  \R$, defined on the set ${\rm Cl}(M)$ of all closed subsets\footnote{By a standard procedure, \cite[Remark~5.4.2]{PR14} or \cite[Remark~1.9]{dickstein2024symplectic}, one can define $\tau$ on open subsets as well. We do not include this definition in the present paper because it is not used in our applications. } of $M$, which satisfies the following axioms:
	\begin{enumerate}[(1)]
		\item \label{qm-normalization}$\tau(M)=1$.
                 \item \label{qm-monotonicity}If $A \subset B$, then $\tau(A) \leq \tau(B)$.
                 \item \label{qm-vanishing} If A is Reeb invariant and displaceable by an element of ${\rm Cont}_0(M),$ then $\tau(A)=0.$
                \item\label{qm-subadditivity} Let $F:M\to\R^N$ be a contact $\alpha$-involutive map and $X ,Y\subset \R^N$ closed subsets. Denote $A:= F^{-1}(X)$ and $B:=F^{-1}(Y)$. If at least one of $A, B$ is displaceable by an element of $\widetilde{\rm Cont}_0(M),$ then
                \[ \tau(A\cup B)= \tau(A)+ \tau(B). \]
\end{enumerate}
\end{dfn}

\begin{theorem}\label{cor:qm} Let $(M,\xi)$ be a closed contact manifold that is strongly fillable by a weakly\textsuperscript{+} monotone symplectic manifold $W$ such that the unit in ${\rm SH}_\ast(W)$ is not eternal. Then, for every contact form $\alpha$ of $\xi$, there exists a contact quasi-measure $\tau$ with respect to $\alpha$. \end{theorem}

In fact, we will show that any contact quasi-state induces a contact quasi-measure. In particular, for any contact form $\alpha$ on $(M, \xi)$, we obtain a quasi-measure from $\zeta_{\alpha}$ mentioned above, which is denoted by $\tau_{\alpha}$.

\begin{proof}[Proof of Theorem~\ref{thm:bft2}]
Assume the contrary, i.e. that there exists a contact involutive map $F:M\to\R^N$ whose all the fibres are displaceable by ${\rm Cont}_0(M).$  By definition, $F$ is contact involutive if it is contact $\alpha$-involutive for some contact form $\alpha$ on $M$. Since $M$ is compact and the fibres of $F$ are all displaceable, there exist closed sets $A_j= F^{-1}(X_j), j=1,\ldots, k$ that are displaceable by ${\rm Cont}_0(M)$ and whose interiors cover $M$. Now, under our hypothesis, Theorem \ref{cor:qm} yields a contact quasi-measure $\tau_{\alpha}$. Then, using \eqref{qm-normalization}, \eqref{qm-vanishing}, and \eqref{qm-subadditivity} in Definition~\ref{dfn-qm}, we get the following contradiction
\[  1=\tau_\alpha(M) = \tau_\alpha(A_1\cup\cdots\cup A_k) = \tau_\alpha(A_1)+\cdots+\tau_\alpha(A_k)=0.\]
This completes the proof.
\end{proof}

\begin{proof}[Proof of Theorem~\ref{thm-bft}]
Theorem~\ref{thm-bft} is a special case of Theorem~\ref{thm:bft2} because the unit $e\in {\rm SH}_\ast(W)$ is not eternal if $W$ is a Liouville domain with non-vanishing symplectic homology by Proposition \ref{prop:uniteternal}. 
\end{proof}

\section{A computational example} \label{sec-ex} In this section, we compute the contact spectral invariant $c(h, \theta)$ in a concrete case.
Consider $(S^*\mathbb S^n, \xi_{\rm can} = \ker \alpha|_{\rm can})$, for $n>1$ odd, the unit co-sphere bundle with respect to the canonical contact structure. It arises as the boundary of a Liouville domain - the unit co-disk bundle of the $n$-sphere $\mathbb S^n$ denoted by $D^*\mathbb S^n = D^*_{g_{\rm std}} \mathbb S^n$, with respect to the standard round metric $g_{\rm std}$ on $\mathbb S^n$.  By \cite{Abbondandolo-Schwarz,Abbondandolo-Schwarz-2,Ziller}, the symplectic homology $D^\ast \mathbb{S}^n$ in $\mathbb{Z}_2$-coefficients is given by
\[{\rm SH}_k (D^\ast\mathbb{S}^n)\cong \left\{ \begin{matrix} \mathbb{Z}_2 & \text{for } k\in(n-1)\mathbb{Z}_{\geqslant 0} \text{ or } k\in (n-1)\mathbb{Z}_{\geqslant 0}+ n\\ 0 & \text{otherwise.} \end{matrix}\right. \] 
The symplectic homology ${\rm SH}_\ast(D^\ast\mathbb{S}^n)$ as an algebra is isomorphic to $\bigwedge(a)\otimes\mathbb{Z}_2[u]$ (which can be rewritten as $\mathbb Z_2[u,a]/{(a^2 =0)}$), up to a shift in grading \cite{cohen2004loop,cieliebak2023loop}. More precisely, the elements of ${\rm SH}_\ast(D^\ast\mathbb{S}^n)$ are the linear combinations of $a u^k$ and $u^k$ where $k\in\mathbb{N}\cup \{0\}$, $a\in {\rm SH}_0(D^\ast\mathbb{S}^n),$ and $u\in {\rm SH}_{2n-1}(D^\ast\mathbb{S}^n).$ The pair-of-pants product is given by
\[u^k\ast u^\ell= u^{k+\ell}, \quad (au^k)\ast (au^\ell)= 0, \quad (au^k)\ast(u^\ell)= au^{k+\ell}.\]
 Recall that the pair-of-pants product comes with the following shift in grading:
\[{\rm SH}_k(D^\ast\mathbb{S}^n)\otimes {\rm SH}_\ell(D^\ast\mathbb{S}^n) \to {\rm SH}_{k+\ell-n}(D^\ast\mathbb{S}^n).\] 

We compute the contact spectral invariant for $h =0$ (hence, ${\rm osc}_R(h) = 0$). It is well-known that closed Reeb orbits on $S^*\mathbb S^n$ precisely correspond to the closed geodesics on $(S^n, g_{\rm std})$, which implies that the spectrum $\mathcal{S}_0$ is equal to $\{2\pi m \, |\, m \in \Z\}$.

By \cite[Proposition 4.4]{U-Viterbo}, we have
\[{\rm HF}_k(2\pi m+\varepsilon)\cong\left\{ \begin{matrix}  \mathbb{Z}_2 & k\in\{\ell(n-1), \ell(n-1)+n\:|\: \ell\in \Z\cap[0, 2m]\} \\ 0 & \text{otherwise}\end{matrix}\right.\]
for $m\in\mathbb{N}$ and any $\varepsilon\in(0,2\pi)$. Moreover, the canonical map (that comes from the composition of Floer continuation maps for consecutive $m$), 
\begin{equation} \label{can-map}
{\rm HF}_k(2\pi m+\varepsilon)\to {\rm SH}_k(D^\ast \mathbb{S}^n)
\end{equation}
is an isomorphism if $2m> k-1$. Therefore,
\begin{equation}\label{eq:computation} c(h=0, u^k)=-2\pi\cdot \left\lceil\frac{k}{2}\right\rceil \quad\text{and}\quad  c(h=0, au^k)=-2\pi\cdot \left\lceil\frac{k}{2}\right\rceil,\end{equation}
for $k>0.$ The formulae \eqref{eq:computation} also hold for $k=0$, but an extra argument is needed to ensure that the elements $a$ and the unit $e=u^0$ do not ``appear'' for negative $m$. Since $D^\ast\mathbb{S}^n$ is a Liouville domain $c(h=0, e)=0$ (this follows from Example~\ref{ex:LD-not-eternal} and the triangle inequality of spectral invariants). Since ${\rm HF}_\ast(-\varepsilon)$ and ${\rm HF}_\ast(\varepsilon)$ can be respectively identified with ${\rm H}_\ast(D^\ast\mathbb{S}^n; \Z_2)$ and ${\rm H}_\ast(D^\ast\mathbb{S}^n, \partial D^\ast\mathbb{S}^n; \Z_2)$ and since the map ${\rm H}_n(D^\ast\mathbb{S}^n; \Z_2)\to{\rm H}_n(D^\ast\mathbb{S}^n, \partial D^\ast\mathbb{S}^n; \Z_2) $ is equal to 0, we get $c(h=0, a)=0$. Consequently, 
\begin{equation}\label{eq:computation2} c(h=0, u^k)=-2\pi\cdot \left\lceil\frac{k}{2}\right\rceil \quad\text{and}\quad  c(h=0, au^k)=-2\pi\cdot \left\lceil\frac{k}{2}\right\rceil,\end{equation} 
for all $k\in\mathbb{N}\cup \{0\}.$ Thus, we computed all the contact spectral invariants for $h=0$.

\section{Proof of Theorem \ref{thm:order}} 

Theorem~\ref{thm:order} quickly follows from Theorem~\ref{thm:spectral-properties}.
 
\begin{proof}[Proof of Theorem~\ref{thm:order}]\label{proof:order}
Let $W$ be a weakly\textsuperscript{+} monotone strong filling of a closed contact manifold $M$ such that the unit $e\in{\rm SH}_\ast(W)$ is not eternal. We want to show that $M$ is orderable, i.e. that there are no contractible positive loops of contactomorphisms on $M$. Assume the contrary, and let $h:[0,1]\times M\to\R_{>0}$ be a positive contact Hamiltonian that generates a contractible loop of contactomorphisms. By compactness, there exists $\delta>0$ such that $h>\delta.$ Then,
\[ c(0, e)< c(0,e)+\delta= c(\delta, e)\leqslant c(h, e)= c(0,e). \]
Here, we used the shift, the monotonicity, and the descent properties of the contact spectral invariants. This contradiction completes the proof.
\end{proof}

\section{Proofs of Theorems \ref{thm:she-conj} and \ref{thm-tran-eternal}}

To start, let us introduce spectral-type invariants $\sigma(h, \theta)$ associated to a contact Hamiltonian $h:[0,1]\times M\to\R$ and an element $\theta\in{\rm H}_\ast(W, \partial W)$ of the singular homology. The invariant $\sigma$ is, in a way, opposite to the spectral invariant $c$ from Definition~\ref{def:spectral}. Indeed, while $c$ measures the first instance at which an element of ${\rm SH}_\ast(W)$ appears, the invariant $\sigma$ measures the smallest instance at which an element of ${\rm H}_\ast(W,\partial W)$ disappears. Here is the precise definition of $\sigma.$

\begin{dfn}\label{def:anti-spectral}
Let $h:[0,1]\times M\to\R$ be a contact Hamiltonian on a contact manifold $M$ that is strongly filled by a weakly\textsuperscript{+} monotone symplectic manifold $W$. Let $\theta\in{\rm H}_\ast(W,\partial W).$ Then, the spectral invariant $\sigma(W, h, \theta)$ is defined by
\begin{equation*}\label{eq:spectral} \sigma(W, h, \theta):= -\inf\left\{ \eta\in[-\min h, +\infty)\:\bigg|\; \theta\in {\rm ker} \bigg({\rm H}_{\ast+n}(W, \partial W)\to {\rm HF}_\ast(\eta\#h) \bigg) \right\}.
\end{equation*}
\end{dfn}

When there is no danger of confusion, we write shortly $\sigma(h,\theta)$ instead of $\sigma(W, h, \theta).$ Notice that we require $\eta\# h \geqslant 0$ for the map ${\rm H}_{\ast+n}(W, \partial W)\to {\rm HF}_\ast(\eta\#h)$ to be well defined. The condition $\eta\#h \geqslant 0$ is equivalent to $\eta \geqslant -\min h.$

\begin{prop} \label{thm-anti-si} 
Let $W$ be a weakly\textsuperscript{+} monotone strong filling of a contact manifold $M$, and let $\theta\in {\rm H}_\ast(W, \partial W)$ be an element that is mapped to 0 via the canonical morphism ${\rm H}_\ast(W, \partial W)\to {\rm SH}_\ast(W)$. Then, the following holds for all contact Hamiltonians $h,g:[0,1]\times M\to\R$:
\begin{enumerate}[(1)]
\item {\rm (Spectrality)} $\sigma(h, \theta)\in (-\mathcal{S}_h)\cup\{\min h\}.$
\item {\rm (Shift)} $\sigma(s\# h, \theta)=\sigma(h, \theta)+ s$ for all $s\in \R.$
\item {\rm (Monotonicity)} $\sigma(h, \theta)\leqslant \sigma(g, \theta)$ if $h\leqslant g.$
\item {\rm (Stability)} $-{\rm osc}_R(g,h)\leqslant\sigma(h, \theta) -\sigma(g, \theta)\leqslant {\rm osc}_R(h,g).$
\end{enumerate}
\end{prop}
\begin{proof}
The proof is analogous to the proof of Theorem~\ref{thm:spectral-properties}.
\end{proof}

\begin{lemma}\label{lem:no-trans}
Let $W$ be a weakly\textsuperscript{+} monotone strong filling of a closed contact manifold $M$. Assume there exists $\varphi\in{\rm Cont}_0(M)$ with no translated points. Then, the following statements hold:
\begin{enumerate}[(1)]
\item The canonical map ${\rm SH}^{-\ast}(W)\to{\rm SH}_\ast(W)$ is an isomorphism.
\item The canonical map ${\rm H}_{\ast+n}(W, \partial W)\to {\rm SH}_\ast(W)$ is not injective.
\end{enumerate} 
\end{lemma}
\begin{proof}
Let $h:[0,1]\times M\to\R$ be a contact Hamiltonian such that $\varphi_h^1=\varphi.$ Since, for each $a\in \R$, there exist $\eta, \eta'\in\R$ such that $\eta\#h\leqslant a\leqslant \eta'\#h,$ we have
\[ {\rm SH}^{-\ast}(W) = \underset{\eta}{\lim_{\longleftarrow}}\: {\rm HF}_\ast(\eta\#h)\quad\text{and}\quad {\rm SH}_{\ast}(W) = \underset{\eta}{\lim_{\longrightarrow}}\: {\rm HF}_\ast(\eta\#h).\]
The limits are taken with respect to the continuation maps. Since $\varphi_h^1$ has no-translated points, the groups ${\rm HF}_\ast(\eta\#h)$ are all isomorphic via the zig-zag isomorphisms. This, however, does not imply that the continuation maps between them (when defined) are isomorphisms (and the proof is thus not finished yet). By Proposition \ref{prop:zigzagSH}, the following diagrams commute for $\Delta\geqslant {\rm osc}_Rh$ and $\eta'\leqslant \eta:$
\begin{equation}\label{eq:delta}
\begin{tikzcd}
&{\rm HF}_\ast((\eta+\Delta)\#h)& \\
        {\rm HF}_\ast(\eta'\#h)\arrow{rr}{\mathcal{\cong}}\arrow{ru} & & {\rm HF}_\ast(\eta\#{h}),\arrow{lu} 
\end{tikzcd}
\end{equation}
\begin{equation}\label{eq:nabla}
\begin{tikzcd} 
        {\rm HF}_\ast(\eta'\#h)\arrow{rr}{\mathcal{\cong}} & & {\rm HF}_\ast(\eta\#{h})\\
        &{\rm HF}_\ast((\eta'-\Delta)\#h).\arrow{lu}\arrow{ru}& 
\end{tikzcd}
\end{equation}
In these diagrams, the horizontal arrows are the zig-zag isomorphisms and all the other arrows are the continuation maps. Fix $\Delta\geqslant {\rm osc}_R(h)$ and denote $G_k:= {\rm HF}_\ast((k\Delta)\# h).$ For $k\leqslant \ell,$ denote by $\Phi_k^\ell: G_k\to G_\ell$ the continuation map. The diagram \eqref{eq:delta} implies ${\rm im}\Phi_{k_1}^\ell={\rm im}\Phi_{k_2}^\ell$ for $k_1, k_2\leqslant\ell.$ Denote this image by $I_\ell.$ We have $\Phi_k^\ell(I_k)= \Phi_k^\ell(G_k)= I_\ell$ for $k\leqslant \ell.$ 

Let us now show that the restriction of $\Phi_k^\ell$ to $I_k$ is in fact injective. Assume the contrary. Then, there exists a non-zero $a\in I_k$ such that $\Phi_k^\ell(a)=0.$ By the definition of $I_k$, there exists $b\in I_{k-1}$ such that $\Phi_{k-1}^k(b)=a.$ Since $b\not \in \ker \Phi_{k-1}^k$ and $b\in \ker \Phi_{k-1}^\ell$, we have $\ker \Phi_{k-1}^k\not = \ker \Phi_{k-1}^\ell.$ On the other hand, the diagram \eqref{eq:nabla} implies $\ker \Phi_{k-1}^{\ell_1}\ = \ker \Phi_{k-1}^{\ell_2}$ for all $\ell_1, \ell_2 > k-1.$ This is a contradiction that proves $\left.\Phi_{k}^\ell\right|_{I_k}$ is injective. Hence, $\Phi_k^\ell$ maps $I_k$ to $I_\ell$ isomorphically.

Now, we have a directed system $(\{I_k\}_k, \{\Phi_k^\ell\}_{k\leqslant\ell})$ whose morphisms are all isomorphisms. In addition, $I_k=\Phi_{k-1}^k(G_k).$ Therefore,
\[ {\rm SH}^{-\ast}(W)=\lim_{\longleftarrow} G_k = \lim_{\longleftarrow} I_k \cong \lim_{\longrightarrow} I_k = \lim_{\longrightarrow} G_k = {\rm SH}_\ast(W).\] 
This proves the first part of the lemma. Let us now prove the second part. Assume by contradiction that the map ${\rm H}_{\ast+n}(W, \partial W)\to {\rm SH}_\ast(W)$ is injective. Since ${\rm SH}^{-\ast}(W)\to {\rm SH}_\ast(W)$ is an isomorphism, this implies that the map ${\rm SH}^{-\ast}(W)\to {\rm H}_{\ast+n}(W, \partial W)$ has to be an epimorphism (in fact, an isomorphism). This, however, cannot be true because the map ${\rm SH}^{-\ast}(W)\to {\rm H}_{\ast+n}(W, \partial W)$ factors through the (non-surjective) map ${\rm H}_{\ast+n}(W)\to {\rm H}_{\ast+n}(W, \partial W).$ To see the latter, notice that ${\rm H}_{2n}(W)=0\not={\rm H}_{2n}(W, \partial W).$ This completes the proof of the lemma.
\end{proof}

\begin{proof}[Proof of Theorem~\ref{thm:she-conj}]
Let $W$ be a weakly\textsuperscript{+} monotone strong filling of $M$. Without loss of generality, we may assume that $\alpha$ is the restriction to $M$ of a (locally defined) Liouville form of $W$. Denote by $A\in\R_{>0}$ the minimal period of a closed Reeb orbit on $M$. Since $\abs{\varphi}^{\rm osc}_{\alpha}<A,$ there exists a contact Hamiltonian $h:[0,1]\times M\to \R$ such that $[\{\varphi^t_h\}]=\varphi$ and such that $\max{h} - \min{h}< A.$
By Lemma~\ref{lem:no-trans}, there exists a non-zero $\theta\in{\rm H}_\ast(W, \partial W)$ which is mapped to 0 via the canonical map ${\rm H}_{\ast+n}(W, \partial W)\to {\rm SH}_\ast(W).$ Hence, $\sigma(h, \theta)\in (-\infty, \min h].$
The property (4) in Theorem \ref{thm-anti-si} implies
\[ \min h=-{\rm osc}_R(0,h)\leqslant\sigma(h, \theta) -\sigma(0, \theta)\leqslant {\rm osc}_R(h,0)=\max h.  \]
Assume now, by contradiction, that $\varphi$ has no translated points. Then, $\sigma(h, \theta)= \min h.$ Therefore,
$\min h-\max h\leqslant \sigma(0, \theta)\leqslant 0. $ On the other hand, since $\theta\not=0$, we have $\sigma(0, \theta)\leqslant -A.$ Hence, $ \max h-\min h\geqslant A$. This contradicts assumptions of the theorem and finishes the proof.
\end{proof}

\begin{proof} [Proof of Theorem \ref{thm-tran-eternal}] The proof follows directly from Lemmas \ref{prop:eternal} and \ref{lem:no-trans}. 
\end{proof}

\section{Proof of Theorem~\ref{thm:spectral-properties}} \label{sec-spectral-prop}

The following lemma is used in the proof of the spectrality property.

\begin{lemma}\label{lem:slope-modification}
Let $W$ be a weakly\textsuperscript{+} monotone strong filling of a closed contact manifold $M$ and let $h:[0,1]\times M\to\R$ be a contact Hamiltonian. Assume $[a,b]\subset \R\setminus\mathcal{S}_h.$ Then, the image ${\rm im}\left( {\rm HF}_\ast(\eta\#h)\to {\rm SH}_\ast(W) \right)$ is the same for all $\eta\in[a,b].$
\end{lemma}
\begin{proof}
Since $[a,b]$ is disjoint from $\mathcal{S}_h$, the family $s\#h, s\in[a,b]$ is a smooth $s$-family of admissible contact Hamiltonians. Therefore, there exists a zig-zag isomorphism ${\rm HF}_\ast(\eta_0\#h)\to {\rm HF}_\ast(\eta_1\#h)$ for any $\eta_0, \eta_1\in[a,b].$ Now, Proposition~\ref{prop:zigzagSH} implies the lemma.
\end{proof}

The next proposition proves the descent property.
\begin{prop}
Let $W$ be a weakly\textsuperscript{+} monotone strong filling of a closed contact manifold $M$. Then, $c(h,\theta)=c(g,\theta)$ for all $\theta\in{\rm SH}_\ast(W)$ whenever the contact Hamiltonians $h,g$ give rise to the same class in $\widetilde{\rm Cont}_0(M).$
\end{prop}
\begin{proof}
Let $h,g:[0,1]\times M\to\R$ be two contact Hamiltonians such that $[\{ \varphi^t_h \}]= [\{ \varphi^t_g \}]$ in $\widetilde{\rm Cont}_0(M)$. Then, $\bar{h}\# g$ is the contact Hamiltonian that generates the contractible loop $(\varphi_h^t)^{-1}\varphi_g^t$ of contactomorphisms. Therefore, there exists a smooth $s$-family $f^s:[0,1]\times M\to\R$ of contact Hamiltonians such that $f^0=0,$ $ f^1=\bar{h}\#g $, and $\varphi_{f^s}^1={\rm id}$ for all $s$. Let $\eta\in\R\setminus \mathcal{S}_h.$ Then, $\{ \eta\# h\# f^s \}_{s\in[0,1]}$ is a smooth $s$-family of admissible contact Hamiltonians. Denote by $\mathcal{B}: {\rm HF}_\ast(\eta\#h)\to {\rm HF}_\ast(\eta\#g)$ the zig-zag isomorphism associated to it. By Proposition~\ref{prop:zigzagSH}, the following diagram commutes:
\[
\begin{tikzcd}
        &{\rm SH}_\ast(W)& \\
        {\rm HF}_\ast(\eta\#h)\arrow{rr}{\mathcal{B}}\arrow{ru} & & {\rm HF}_\ast(\eta\#g).\arrow{lu} 
    \end{tikzcd}
\] 
In particular, $\theta\in {\rm im}({\rm HF}_\ast(\eta\#h)\to {\rm SH}_\ast(W))$ if, and only if, $\theta\in {\rm im}({\rm HF}_\ast(\eta\#g)\to {\rm SH}_\ast(W))$, because these two images are the same. Hence, $c(h,\theta)=c(g, \theta)$, and the proof is finished.
\end{proof}

\begin{proof}[Proof of Theorem~\ref{thm:spectral-properties}]
The shift and the monotonicity properties follow directly from the definition. Let us now show the spectrality property. Assume the contrary, i.e. that $-c(h,\theta)=:c\in \R\setminus\mathcal{S}_h$ for some $h$ and $\theta$. 
The set $\mathcal{S}_h$ is a closed subset of $\R.$ Therefore, there exists $\varepsilon>0$ such that $[c-\varepsilon, c+\varepsilon]\subset \R\setminus \mathcal{S}_h.$ By Lemma~\ref{lem:slope-modification}, either $\theta \in {\rm im}\left( {\rm HF}_\ast(\eta\#h)\to {\rm SH}_\ast(W) \right)$ for all $\eta\in [c-\varepsilon, c+\varepsilon]$ or for none. In either event, this implies $-c(h,\theta)\not=c$ and finishes the proof of the spectrality property by contradiction.

Now, we prove the stability property.  If $\Delta\geqslant {\rm osc}_R(g,h),$ we have
\begin{align*}
((\eta + \Delta)\#h)_t = &\: \eta +\Delta + h_t\circ\varphi^{(\eta+\Delta)t}_R\\
\geqslant &\: \eta + {\rm osc}_R(g,h) + g_t\circ \varphi_R^{\eta t} - g_t\circ \varphi_R^{\eta t} + h_t\circ\varphi_R^{(\eta+\Delta)t} \\
= &\: (\eta\# g)_t + {\rm osc}_R(g,h) - \left( g_t\circ \varphi_R^{\eta t} - h_t\circ\varphi_R^{(\eta+\Delta)t}\right)\\
\geqslant &\: (\eta\# g)_t.
\end{align*}
Therefore, $c(g,\theta)\leqslant c(h,\theta) + {\rm osc}_R(g,h).$ By symmetry, $c(h,\theta)\leqslant c(g,\theta) + {\rm osc}_R(h,g).$ These two inequalities imply the stability property.

Let us now show the triangle inequality. Reparametrizing the time interval does not change the induced class in $\widetilde{\rm Cont}_0(M)$. In addition, ${\rm osc}_R(\mu'(t) h_{\mu(t)})$ can be made arbitrarily close to ${\rm osc}_R h$ by choosing a reparametrization $\mu:[0,1]\to[0,1]$ that is constant near 0 and near 1. Therefore, by the descent property, we may assume without loss of generality that $h_t$ and $g_t$ are equal to 0 for $t$ close to 0 or 1. Denote by ${\rm I}(f)$ the image ${\rm im} ({\rm HF}_\ast(f)\to {\rm SH}_\ast(W))$.  Let $\eta_1, \eta_2\in\R$ and let $\theta_1\in {\rm I}(\eta_1\#h), \theta_2\in {\rm I}(\eta_2\#g).$ Using Theorem~\ref{thm:product}, we will show
\[\theta_1\ast\theta_2\in  {\rm I}((\eta_1+\eta_2+\Upsilon)\#(h\bullet g)),\]
where $ \Upsilon:=2\cdot \max\{ {\rm osc}_R h, {\rm osc}_Rg\}.$ Theorem~\ref{thm:product}, however, cannot be directly applied to $\eta_1\#h$ and $\eta_2\#g$ because the concatenation $(\eta_1\#h)\bullet (\eta_2\#g)$ is in general not smooth (not even continuous). For this reason, we ``reparametrize'' $\eta_1$ and $\eta_2$ first. Let $a, b:[0,1]\to \R$ be smooth functions such that $a_t=b_t=0$ for $t$ near 0 or 1 and such that
\[ \int_0^1 a_t dt= \eta_1\quad\text{and}\quad \int_0^1 b_t dt=\eta_2 .\]
We see $a$ and $b$ as time-dependent,  but $M$-independent, contact Hamiltonians on $M$. Denote $c:= a\bullet b$, that is $c_t= 2a_{2t}$ for $t\in[0,\frac{1}{2}]$ and $c_t=2 b_{2t-1}$ for $t\in[\frac{1}{2}, 1]$. Notice that $\int_0^1 c_t dt=\eta_1+\eta_2.$ A straightforward computation shows
\begin{equation}\label{eq:ineq}
 (a\# h)\bullet (b\#g)\leqslant (c + \Upsilon)\#(h\bullet g).
\end{equation}
Proposition~\ref{prop:zigzagSH} implies 
\begin{equation}\label{eq:Is}
\begin{split}
&{\rm I}(a\#h)={\rm I}(\eta_1\#h),\quad {\rm I}(b\#g)={\rm I}(\eta_2\#g),\quad\text{and}\\
& {\rm I}((c + \Upsilon)\#(h\bullet g)) = {\rm I}( (\eta_1+\eta_2+\Upsilon)\#(h\bullet g)). 
\end{split}
\end{equation}
Theorem~\ref{thm:product}, together with \eqref{eq:ineq} and \eqref{eq:Is}, implies that for $\theta_1\in {\rm I}(\eta_1\#h)$ and $\theta_2\in {\rm I}(\eta_2\#g)$ we have
$ \theta_1\ast\theta_2\in {\rm I}((a\#h)\bullet (b\#g))\subset {\rm I}((\eta_1+\eta_2+\Upsilon)\#(h\bullet g)). $ Consequently, 
\begin{align*}
c(h,\theta_1) + c(g, \theta_2) &\leqslant  c(h\bullet g,\theta_1\ast \theta_2) + \Upsilon 
= c(h\bullet g,\theta_1\ast \theta_2) + 2\cdot \max\{ {\rm osc}_R h, {\rm osc}_Rg\}.
\end{align*}
Since $g\bullet h$ and $h\# g$ give rise to the same class in $\widetilde{\rm Cont}_0(M)$, the descent property implies $c(g\bullet h)=c(h\#g).$ This finishes the proof.
\end{proof}

\section{Proof of Theorem~\ref{thm-qs}}\label{sec:quasi-states}

Let  $W$ be a weakly\textsuperscript{+} monotone strong filling of a contact manifold $M$ and let $\alpha$ be the induced contact form on $M$. Assume that the unit $e\in {\rm SH}_\ast(W)$ is not eternal.

\begin{dfn}\label{def:zeta}
	For a smooth function $h\in C^\infty(M)$, define
	\[\zeta_\alpha(h):=\liminf_{k\to\infty}\frac{c(kh, e)}{k},\]
	where $c$ is as in Section~\ref{sec:spectral} and $k$ in the limit inferior goes through $(0, +\infty).$
\end{dfn}

We show here that $\zeta_\alpha$ is a partial contact quasi-state with respect to $\alpha$. We start by proving that the limit inferior in Definition~\ref{def:zeta} is actually a limit if the smooth function $h$ is Reeb invariant. 

\begin{lemma}
	The limit $\displaystyle\lim_{k\to \infty}\frac{c(kh, e)}{k}$ exists and is finite for all Reeb invariant functions $h\in C^\infty(M).$
\end{lemma}
\begin{proof}
	The triangle inequality form Remark~\ref{rmk:spectral-strict} implies that the function 
	\[\R_{>0}\to \R\quad :\quad k\mapsto c(kh, e)\] 
is superadditive. By Fekete's superadditivity lemma, $\lim_{k\to\infty}\frac{c(kh, e)}{k}$ exists. Since $\frac{c(kh, e)}{k}$ is bounded, as a consequence of the stability property from Remark~\ref{rmk:spectral-strict}, this finishes the proof.
\end{proof}

The following proposition is crucial for proving the vanishing property of the partial contact quasi-state $\zeta_\alpha.$

\begin{prop}\label{prop:yaron}
	Let $U$ be a Reeb invariant subset and let $\varphi_f^t:M\to M$ be a contact isotopy such that $\varphi_f^1(U)\cap U=\varnothing.$ Then, for all $h$ with ${\rm supp}\:h\subset U$, we have\footnote{See \eqref{eq:spectrum} on page~\pageref{eq:spectrum} for the definition of $\mathcal{S}_f.$ }
	\[ \mathcal{S}_{h\#f} =\mathcal{S}_f.\]
	In other words, the sets of shifts for translated points of the contactomorphisms $\varphi_h^1\varphi_f^1$ and $\varphi_f^1$ are equal. In fact, the sets of translated points for a given shift coincide for these contactomorphisms.
\end{prop}
\begin{proof}
	It is enough to show that the sets of translated points with shift $\eta$ coincide for the contactomorphisms $\varphi^1_h\varphi^1_f$ and $\varphi_f^1$. We first show ${\rm Fix}(\varphi_R^{-\eta}\varphi^1_h\varphi^1_f)= {\rm Fix}(\varphi_R^{-\eta}\varphi^1_f).$
	Both $\varphi^{-\eta}_R\varphi_h^1\varphi_f^1$ and $\varphi^{-\eta}_R\varphi_f^1$ have no fixed points in $U$. Indeed, if $x\in U$ then $\varphi^1_f(x)\not \in U$ and $\varphi^1_h\varphi_f^1(x)=\varphi_f^1(x)\not \in U.$ Therefore, by the Reeb invariance  of $U$ and its complement, $\varphi_R^{-\eta}\varphi_h^1\varphi_f^1(x)\not\in U$ and $ \varphi_R^{-\eta}\varphi_f^1(x)\not\in U .$ Hence, $x$ cannot be a fixed point of either $\varphi_R^{-\eta}\varphi_h^1\varphi_f^1$ or $\varphi_R^{-\eta}\varphi_f^1$. 
	
	If, on the other hand, $x\not\in U$ then we have the following sequence of equivalent statements:
\begin{align*}
 	x= \varphi_R^{-\eta}\varphi_f^1(x)\quad&\Longleftrightarrow\quad \varphi_R^\eta(x)=\varphi_f^1(x)\quad \Longleftrightarrow\quad (\varphi_h^1)^{-1}\varphi_R^\eta(x)= \varphi_f^1(x)\\
	&\Longleftrightarrow\quad \varphi_R^\eta(x)=\varphi_h^1\varphi_f^1(x) \quad \Longleftrightarrow\quad 
	 x= \varphi_R^{-\eta}\varphi_h^1\varphi_f^1(x).
\end{align*}
Thus, ${\rm Fix}(\varphi_R^{-\eta}\varphi^1_h\varphi^1_f)= {\rm Fix}(\varphi_R^{-\eta}\varphi^1_f).$ Let us now prove $\kappa^1_{h\#f}(x)= \kappa^1_f(x)$ for all $x\in {\rm Fix}(\varphi_R^{-\eta}\varphi^1_f).$ The conformal factor $\kappa^1_{h\#f}$ can be expressed as
\[ \kappa^1_{h\#f}(x) = \kappa^1_h(\varphi^1_f(x))\cdot \kappa^1_f(x). \]
If $x \in {\rm Fix}(\varphi_R^{-\eta}\varphi^1_f),$ then $x\not \in U$ and $\varphi_f^1(x)= \varphi_R^\eta(x).$ Since $U^c$ is Reeb invariant, we have $\varphi_f^1(x)\not \in U.$ Consequently, $\kappa_h^1(\varphi^1_f(x))=1.$ Therefore, $\kappa^1_{h\#f}(x) = \kappa^1_f(x)$ for all $x \in {\rm Fix}(\varphi_R^{-\eta}\varphi^1_f).$ This finishes the proof of the proposition.
\end{proof}

\begin{cor}\label{cor:vanishing}
	In the situation of Proposition~\ref{prop:yaron}, we have $c(h\#f, \theta)=c(f, \theta)$ for all $\theta\in{\rm SH}_\ast(W).$ 
\end{cor}
\begin{proof}
	Proposition~\ref{prop:yaron} implies that $\eta\#(s\cdot h)\#f$, for $s\in[0,1]$, is a smooth family of admissible contact Hamiltonians for all $\eta\not\in\mathcal{S}_f$. Therefore (see Section~\ref{sssec-zz}), there exists a zig-zag isomorphism ${\rm HF}_\ast(\eta\# f)\to {\rm HF}_\ast(\eta\# h\# f)$ that commutes with the canonical maps ${\rm HF}_\ast(g)\to{\rm SH}_\ast(W).$ Proposition~\ref{prop:zigzagSH} now implies $c(h\#f, \theta)= c(f, \theta).$
\end{proof}

Next, we introduce the conjugation isomorphisms. They are used in the proof of the conjugation property of $\zeta_\alpha$. Denote by ${\rm Symp}^*(W)$ the group of symplectomorphisms $\psi:\widehat{W}\to\widehat{W}$ of the completion $\widehat{W}$ that preserve the Liouville form outside of a compact subset. For each $\psi\in{\rm Symp}^*(W)$ there exists a contactomorphism $\varphi:M\to M$ such that $\psi(x,r) = \left(\varphi(x), \frac{r}{\kappa_\varphi(x)} \right)$ on the conical end for $r$ large enough. Recall that $\kappa_\varphi$ denotes the conformal factor of $\varphi,$ i.e. the positive function $M\to\R_{>0}$ determined by $\varphi^\ast \alpha= \kappa_\varphi\cdot \alpha.$ We call the contactomorphism $\varphi$ \emph{the ideal restriction} of $\psi$ and denote $\varphi:=\psi_\infty$. The key ingredient in the proof of Lemma~\ref{lem:conjugation} below is the existence of the conjugation isomorphisms (cf. \cite[Section~5]{U-selective})
\[ \mathcal{C}_\psi : {\rm HF}_\ast(h) \to {\rm HF}_\ast((\psi_\infty)^\diamondsuit h) \]
for all $\psi\in{\rm Symp}^*(W)$ and for admissible contact Hamiltonians $h_t:M\to \R.$ Here, $(\psi_\infty)^\diamondsuit h:= (h\circ \psi_\infty)/{\kappa_{\psi_\infty}}.$ The conjugation isomorphism $\mathcal{C}_{\psi}$ commutes with the continuation maps, thus inducing an automorphism of symplectic homology (called and denoted the same). Moreover, the conjugation isomorphism $\mathcal{C}_{\psi}:{\rm SH}_\ast(W)\to {\rm SH}_\ast(W)$ preserves the unit.

\begin{lemma}\label{lem:conjugation}
Let $\varphi:M\to M$ be a Reeb invariant contactomorphism that is the ideal restriction of some $\psi\in{\rm Symp}^\ast (W).$ Then, $c(h\circ\varphi, e)=c(h,e)$ for all contact Hamiltonians $h:[0,1]\times M\to \R$.
\end{lemma}
\begin{proof}
Since $\varphi$ is Reeb invariant, we have $\varphi^\diamondsuit(\eta\#h)=\eta\#(\varphi^\diamondsuit h)$ for all contact Hamiltonians $h_t:\partial W\to \R$ and all $\eta\in\R.$ Since $\mathcal{C}_{\psi}$ commutes with the continuation maps and preserves the unit $e$ in ${\rm SH}_\ast(W)$, we have that
	\[e\in{\rm im}\left( {\rm HF}_\ast(\eta\#h)\to {\rm SH}_\ast(W) \right)\]
if, and only if,
	\[e\in{\rm im}\left( {\rm HF}_\ast(\eta\#\varphi^\diamondsuit h)\to {\rm SH}_\ast(W) \right).\]
Consequently, $c(h, e)= c(\varphi^\diamondsuit h, e)=c(h\circ\varphi, e).$
\end{proof}

\begin{proof}[Proof of Theorem~\ref{thm-qs}]
	We show that $\zeta_\alpha$ is a partial contact quasi-state with respect to $\alpha$. The properties \eqref{pqm:normalization}, \eqref{pqm:stability}, and \eqref{pqm:conjugacy} follow directly from the properties of the spectral invariant $c$ in Theorem~\ref{thm:spectral-properties} and Lemma~\ref{lem:conjugation}. The property \eqref{pqm:homogeneous} is an immediate consequence of the definition of $\zeta_\alpha$ via the limit. The first part of the property \eqref{pqm:triangle} (that is, the inequality) follows from the triangle inequality in Theorem~\ref{thm:spectral-properties} and the equation $f\#g=f+g,$ that holds if $\{f,g\}_\alpha=\{f,1\}_\alpha= \{g,1\}_\alpha=0.$ The second part of the property \eqref{pqm:triangle} is a consequence of Corollary~\ref{cor:linear} and the vanishing property \eqref{pqm:vanishing}, which we prove next.
	
	\label{proof:qs}Let us now prove the vanishing property $\eqref{pqm:vanishing}.$ Let $h$ be a Reeb invariant contact Hamiltonian such that there exists a contact isotopy $\varphi_f^t:M\to M$ satisfying $\varphi_f^1({\rm supp}\:h)\cap {\rm supp}\:h=\varnothing.$ Since $h$ is Reeb invariant, the set ${\rm supp}\:h$ is Reeb invariant as well. The triangle inequality from Theorem~\ref{thm:spectral-properties} implies
	\begin{align*} &c(h\#f, e) + c(\bar{f}, e) \leqslant c(h, e) + 2\max\{ {\rm osc}_R(h\#f) , {\rm osc}_R\bar{f}\},\\
	& c(h, e) + c(f, e) \leqslant c(h\# f, e) + 2\max\{ {\rm osc}_Rh , {\rm osc}_R{f}\},
	\end{align*}
where $\bar{f}$ denotes the contact Hamiltonian of the contact isotopy $(\varphi^{t}_f)^{-1}.$
Since $h$ is Reeb invariant, ${\rm osc}_R h=0$ and the expression $2\max\{ {\rm osc}_R(h\#f) , {\rm osc}_R\bar{f}\}$ does not depend on $h$. This is because ${\rm osc}_R(h\#f)= {\rm osc}_R(f)$ holds provided $h$ is Reeb invariant. Denote 
	$K(f):=2\max\{ {\rm osc}_R(h\#f) , {\rm osc}_R\bar{f}\}.$
Then Corollary~\ref{cor:vanishing} implies
	\[ c(f,e) + c(\bar{f}, e) - K(f)\leqslant c(h, e)\leqslant 2 \:{\rm osc}_R f. \]
Since ${\rm supp}\:(kh)= {\rm supp}\:h$ for all positive $k$, we can repeat the same argument for $kh$ instead of $h$ and get
	\[ c(f, e) + c(\bar{f}, e) - K(f)\leqslant c(kh, e) \leqslant 2\:{\rm osc}_R f. \]
	Hence, $c(kh, e)$ is bounded and $\zeta_\alpha(h)=0.$
\end{proof}

\section{The Poisson bracket inequality}

This section proves a Poisson bracket inequality for contact quasi-states from Definition~\ref{def:zeta} in analogy to the Poisson bracket inequality \cite[Proposition~4.6.1]{PR14}. We start by recalling the setting.
Let $M$ be a closed contact manifold with a strong weakly\textsuperscript{+} monotone filling $W$. Assume the unit $e\in{\rm SH}_\ast(W)$ is not eternal. Let $\alpha$ be the induced contact form on $M.$ Denote by $\mathcal{G}\leqslant \widetilde{\rm Cont}_0(M)$ the group of those elements of $ \widetilde{\rm Cont}_0(M)$ that can be represented by Reeb invariant paths of contactomorphisms.
By the descent property of the spectral invariant (see Theorem~\ref{thm:spectral-properties}), $c(\varphi, \theta)$ is well defined for $\varphi\in\widetilde{\rm Cont}_0(M)$ and $\theta\in {\rm SH}_\ast(W).$ Explicitely,
$ c(\varphi, \theta):= c(h, \theta)$ for $[\{\varphi_h^t\}]=\varphi.$
The following proposition summarises the properties of $c(\cdot, e)$ when restricted to $\mathcal{G}.$
\begin{prop}\label{prop:spectralG}
Let $ \varphi, \psi\in \mathcal{G}$ and let $h, g:[0,1]\times M\to\R$ be Reeb invariant contact Hamiltonians such that $[\{\varphi_h^t\}]=\varphi$ and $[\{\varphi_g^t\}]=\psi.$ Then, the following holds:
\begin{enumerate}
\item {\rm (Conjugation invariance)} $c(\varphi^{-1}\psi\varphi, e)= c(\psi, e),$
\item {\rm (Superadditivity)} $ c(\varphi, e) + c(\psi, e) \leqslant c(\varphi\psi, e),$
\item {\rm (Stability)} $\min(h-g)\leqslant c(\varphi, e) - c(\psi, e)\leqslant \max(h-g).$ 
\end{enumerate}
\end{prop}
\begin{proof}
The proposition is a consequence of Theorem~\ref{thm:spectral-properties} and Lemma~\ref{lem:conjugation}.
\end{proof}
The superadditivity in Proposition~\ref{prop:spectralG} implies that the quantity
\begin{equation}\label{eq:q}
q(\varphi):= - c(\varphi, e) - c(\varphi^{-1}, e)
\end{equation}
is non-negative for all $\varphi\in\mathcal{G}.$ Now, we define the \emph{spectral width} of subsets of  $M$ that are either open or closed.

\begin{dfn}\label{def:spectralwidth}
Let $U \subset M$ be an open subset. Denote by $\mathcal{G}(U)\leqslant \mathcal{G}$ the group of those elements of $\widetilde{{\rm Cont}}_0(M)$ that can be represented by a Reeb invariant path of contactomorphisms supported in $U$. The \emph{spectral width} of $U$ is defined to be
\[w(U):=\sup_{\varphi\in \mathcal{G}(U)} q(\varphi).\]
For a closed subset $A\subset M$, the spectral with $w(A)$ is defined as
\[w(A)=\inf_{U\supset A} w(U)\]
where the infimum goes through all the open subsets that contain $A.$
\end{dfn}

The next lemma provides examples of subsets with finite spectral width.

\begin{lemma}\label{lem:finitewidth}
Let $A\subset M$ be either open or closed. If $A$ is displaceable by an element of $\widetilde{\rm Cont}_0(M),$ then $w(A)$ is finite. 
\end{lemma}
\begin{proof}
Since every closed displaceable subset of $M$ has an open neighbourhood that is also displaceable, it is enough to prove the lemma for open subsets $A$. Let $f:[0,1]\times M\to\R$ be a contact Hamiltonian such that $\varphi_f^1(A)\cap A=\varnothing$ and let $\varphi\in \mathcal{G}(A)$ be an arbitrary element. Then, there exists a Reeb invariant contact Hamiltonian $h:[0,1]\times M\to\R$  such that ${\rm supp}\: h\subset A$ and such that $[\{\varphi_h^t\}]=\varphi.$ As in the proof of Theorem~\ref{thm-qs} on page~\pageref{proof:qs}, one can show $L_f\leqslant c(h,e)\leqslant 2{\rm osc}_R f$ where 
\[L_f:= c(f, e) + c(\bar{f}, e) - 2\cdot\max\{ {\rm osc}_R f, {\rm osc}_R \bar{f} \}\]
does not depend on $\varphi$. Similarly, since ${\rm supp}\:\bar{h}={\rm supp}\:{h}\subset A,$ we have $L_f\leqslant c(\bar{h},e)\leqslant 2{\rm osc}_R f.$ This implies
\[ q(\varphi) = -c(\varphi, e) - c(\varphi^{-1}, e) = -c(h, e) - c(\bar{h}, e)\leqslant -2L_f<\infty\]
and finishes the proof.
\end{proof}

\begin{prop}\label{prop:beta}
The limit 
\begin{equation}\label{eq:beta}
\beta(\varphi):=\lim_{k\to\infty}\frac{c(\varphi^k, e)}{k}
\end{equation}
converges for all $\varphi\in\mathcal{G}$. In addition, for $\varphi, \psi\in\mathcal{G}$, the following holds:
\begin{enumerate}
\item \label{item:betaineq}$\abs{\beta(\varphi\psi) - \beta(\varphi) - \beta(\psi)}\leqslant \min\{q(\varphi), q(\psi)\},$
\item \label{item:betastab}$ \abs{\beta(\varphi)-\beta(\psi)}\leqslant \max\abs{f-g}$ if $f, g:[0,1]\times M\to \R$ are Reeb invariant contact Hamiltonians such that $[\{\varphi_f^t\}]=\varphi$ and $[\{\psi_g^t\}]=\psi.$ 
\end{enumerate}
\end{prop}
\begin{proof}
Proposition~\ref{prop:spectralG} (more precisely, the superadditivity property of $c(\varphi, e)$) implies that the sequence $c(\varphi^k, e)$ is superadditive. Hence, by Fekete's superadditivity lemma, $\beta$ is well defined, but a priory may not be finite. 

Now we show \eqref{item:betastab}, which further implies that $\beta$ is always finite. Without loss of generality, assume $f_t=g_t=0$ for $t$ near 0 or $1$. We periodically extend $f$ and $g$ to $\R\times M.$ Notice that $kf_{kt}$ and $kg_{kt}$ induce $\varphi^k$ and $\psi^k$ respectively. The stability from Proposition~\ref{prop:spectralG} implies
\[ \min(kf_{kt}- kg_{kt})\leqslant c(\varphi^k, e)- c(\psi^k, e)\leqslant \max(kf_{kt}- kg_{kt}).\]
Since $\min(kf_{kt}- kg_{kt})=k\cdot \min(f-g)$ and $\max(kf_{kt}- kg_{kt})=k\cdot \max(f-g)$, this proves part \eqref{item:betastab}.
The part \eqref{item:betaineq} follows from abstract Proposition~3.5.3 in \cite{PR14} applied to the map $\mathcal{G}\to\R: \varphi\mapsto -c(\varphi, e).$ The proof is thus completed.
\end{proof}

The following proposition states the Poisson bracket inequality.

\begin{prop}\label{prop:poisson}
Let $f,g:M\to\R$ be Reeb invariant contact Hamiltonians such that 
$S(f, g):= \min\{ w({\rm supp}\: f), w({\rm supp}\: g) \}<\infty. $
Let $\zeta_\alpha$ be as in Definition~\ref{def:zeta}. Then,
\[ \abs{\zeta_\alpha(f+g)-\zeta_\alpha(f) - \zeta_\alpha(g)}\leqslant 2\sqrt{S(f,g)\cdot \max\abs{\{f, g\}_\alpha}}. \]
\end{prop}
\begin{proof}
By a slight abuse of notation, denote by $\varphi_f\in\widetilde{\rm Cont}_0(M)$ the class induced by the path of contactomorphisms $\{\varphi_f^t\}_{t\in[0,1]}.$ Since $f, g,$ and $f+g$ are time-independent and Reeb invariant, we have
$\zeta_\alpha(f)=\beta(\varphi_f),$ $\zeta_\alpha(g)=\beta(\varphi_g),$ and $\zeta_\alpha(f+g)=\beta(\varphi_{f+g}).$ Therefore, by Proposition~\ref{prop:beta},
\begin{align*}
\Pi(f, g):=& \abs{\zeta_\alpha(f+g)-\zeta_\alpha(f) - \zeta_\alpha(g)} = \abs{\beta(\varphi_{f+g})- \beta(\varphi_f) - \beta(\varphi_g)}\\
\leqslant & \abs{\beta(\varphi_{f}\varphi_g)- \beta(\varphi_f) - \beta(\varphi_g)} + \abs{\beta(\varphi_{f+g}) -\beta(\varphi_f\varphi_g)}\\
\leqslant & \min\{ q(\varphi_f), q(\varphi_g) \} + \abs{\beta(\varphi_{f+g}) -\beta(\varphi_f\varphi_g)}\\ \leqslant & S(f, g) + \abs{\beta(\varphi_{f+g}) -\beta(\varphi_f\varphi_g)}.
\end{align*}
To estimate the term $\abs{\beta(\varphi_{f+g}) -\beta(\varphi_f\varphi_g)}$, notice that
\[ g(\varphi_f^t(x))- g(x) = \int_0^t \frac{d}{ds}\left( g(f_s(x)) \right) ds = \int_0^t dg(X_f)(\varphi^s_f(x))ds = \int_0^t \{g, f\}_\alpha(\varphi^s_f(x))ds.\]
Hence, by \eqref{item:betastab} in Proposition~\ref{prop:beta},
\begin{align*}
\abs{\beta(\varphi_{f+g}) -\beta(\varphi_f\varphi_g)}\leqslant& \max\abs{f + g - f - g\circ\varphi^{-t}_f}\\
 = & \max\abs{g\circ\varphi_f^t-g}\leqslant \max\abs{\{f,g\}_\alpha}.
\end{align*}
Consequently,
\begin{equation}\label{eq:piless} \Pi(f,g)\leqslant \max\abs{\{f,g\}_\alpha} + S(f,g). 
\end{equation}
Let $s>0$. By substituting $sf$ and $sg$ respectively for $f$ and $g$ in \eqref{eq:piless}, we get
\[  \Pi(f,g)\leqslant s\cdot \max\abs{\{f,g\}_\alpha} + \frac{1}{s}\cdot S(f,g).\]
By taking the infimum over $s\in\R_{>0}$, we get the desired inequality.
\end{proof}

\begin{cor}\label{cor:linear}
Let $\zeta_\alpha$ be as in Definition~\ref{def:zeta} and let $f, g: M\to\R$ be Reeb invariant contact Hamiltonians such that $\{f,g\}_\alpha=0$ and such that ${\rm supp}\:f$ is contact displaceable. Then, $\zeta_\alpha(f+g)=\zeta_\alpha(f) + \zeta_\alpha(g).$
\end{cor}
\begin{proof}
The corollary follows from Proposition~\ref{prop:poisson} and Lemma~\ref{lem:finitewidth}.
\end{proof}

\section{Proof of Theorem~\ref{cor:qm}} \label{sec:quasi-measures}

Let  $W$ be a weakly\textsuperscript{+} monotone strong filling of a contact manifold $M$ and let $\alpha$ be the induced contact form on $M$. In this section, we show that for any partial contact quasi-state $\zeta$ on $M$ with respect to $\alpha$ the map $\tau:{\rm Cl}(M)\to \R$ given by
\begin{equation}\label{eq:tau} \tau(A):= \inf\left\{ \zeta(h)\:|\: h:M\to[0,1] \text{ Reeb invariant, smooth and } \left.h\right|_A=1 \right\} \end{equation}
is a contact quasi-measure with respect to $\alpha$. The next lemma simplifies the definition of $\tau(A)$ for the case where $A$ is the preimage of a closed subset under a contact $\alpha$-involutive map.

\begin{lemma}\label{lem:tau-simplified}
Let $M$ be a closed contact manifold with a contact form $\alpha$ and let $F:M\to\R^N$ be a contact $\alpha$-involutive map. Let $\zeta$ be a partial contact quasi-state with respect to $\alpha$ and let $\tau$ be as in \eqref{eq:tau}. Then,
\[ \tau(F^{-1}(X))= \inf \left\{\zeta(\mu\circ F)\:|\: \mu:\R^N\to[0,1]\text{ smooth and }\left.\mu\right|_{X}=1\right\} \]
for all closed subsets $X\subset \R^N.$
\end{lemma}
\begin{proof}
Without loss of generality, assume $X\subset F(M)$. It is enough to show that for each smooth $h: M\to[0,1]$ such that $\left.h\right|_{F^{-1}(X)}=1$ and for each $\varepsilon >0$, there exists a smooth function $\mu:\R^N\to[0,1]$ such that $\left.\mu\right|_{X}=1$ and $\mu\circ F\leqslant h+\varepsilon.$ Let $h:M\to [0,1]$ be a smooth Reeb invariant function such that $\left.h\right|_{F^{-1}(X)}=1$. Denote by $\bar{\mu}:F(M)\to[0,1]$ the function defined by  $\bar{\mu}(x):=\inf\left\{h(y)\:|\: y\in F^{ -1}(x)\right\}.$ The function $\bar{\mu}$ need not be continuous in general. However, the set $\bar{\mu}^{-1}((a, +\infty))= \bar{\mu}^{-1}((a, 1])$ is open for all $a\in (-\infty, 1).$ Indeed, the set
\begin{align*}
\bar{\mu}^{-1}((-\infty, a]) & = \left\{x\in F(M)\:\bigg|\: \inf_{y\in F^{-1}(x)} h(y)\leqslant a\right\} = F\left( \{y\in M\:|\: h(y)\leqslant a\} \right)
\end{align*}
is closed as a continuous image of a compact set $\{y\in M\:|\: h(y)\leqslant a\} .$ Therefore, its complement $\barmu^{-1}((a, 1])$ in $F(M)$ is open. Consequently, there exists an open subset $U\subset \R^N$ such that $\barmu^{-1}((1-\varepsilon, 1])= F(M)\cap U.$ Let $\mu:\R^N\to[0,1]$ be a smooth function such that $\mu(p)=1$ for $p\in X$ and $\mu(p)=0$ for $p\in \R^N\setminus U$. Then, $\mu\circ F\leqslant h+\varepsilon$ and the proof is finished.
\end{proof}

\begin{lemma}\label{lem:qsqm}
	Let $\zeta$ be a partial contact quasi-state on $M$ with respect to $\alpha$. Then, $\tau: {\rm Cl}(M)\to \R$ defined by \eqref{eq:tau} is a contact quasi-measure with respect to $\alpha$.
\end{lemma}
\begin{proof}
The properties \eqref{qm-normalization} and \eqref{qm-monotonicity} in Definition~\ref{dfn-qm} are easily verified for $\tau.$ The property \eqref{qm-vanishing} is proven in Lemma~\ref{lem:qm-vanishing} below. Let us show that $\tau$ satisfies property \eqref{qm-subadditivity}. Let $F:M\to\R^N$ be a contact $\alpha$-involutive map and let $X, Y\subset \R^N$ be two closed subsets. Denote $A:=F^{-1}(X)$ and $B:=F^{-1}(Y)$, and assume $B$ is displaceable by an element of $\widetilde{\rm Cont}_0(M).$ Without loss of generality assume $X, Y\subset F(M).$

Since $B$ is contact displaceable, there exists its open neighbourhood $U\subset M$ which is also contact displaceable. Let $\mu_a, \mu_b:\R^N\to[0,1]$ be smooth functions such that $\left.\mu_a\right|_X=1,$ $\left.\mu_b\right|_Y=1,$ and such that ${\rm supp}\:\mu_b\circ F\subset U.$  Let $\varepsilon>0$ and let $\mu:\R^N\to[0,1]$ be a smooth function such that $\abs{\mu-\min\{1, \mu_a+\mu_b\}}\leqslant \varepsilon$ and such that $\left.\mu\right|_{X\cup Y}=1$. Since the functions $\mu_a\circ F$ and $\mu_b\circ F$ are Reeb invariant with respect to $\alpha$ and since $F$ is involutive, by the property \eqref{pqm:triangle} of $\zeta$, we get
\[\tau(A\cup B)\leqslant \zeta( \mu\circ F )\leqslant \zeta(\mu_a\circ F + \mu_b\circ F +\varepsilon) = \zeta(\mu_a\circ F) + \zeta(\mu_b\circ F) +\varepsilon. \]
Taking infimum over $\varepsilon>0$, $\mu_a$ and $\mu_b$ yields $\tau(A\cup B)\leqslant \tau(A)+\tau(B).$ By properties \eqref{qm-monotonicity} and \eqref{qm-vanishing}, we have
\[ \tau(A\cup B)\leqslant \tau(A)+\tau(B)=\tau(B)\leqslant \tau(A\cup B).\]
Therefore, $\tau(A\cup B)=\tau(A) + \tau(B)$ and the proof is completed.
\end{proof}

\begin{proof}[Proof of Theorem~\ref{cor:qm}]
	The proof follows from Lemma~\ref{lem:qsqm} and Theorem~\ref{thm-qs}.
\end{proof}
	
The next lemma was used in the proof of Lemma~\ref{lem:qsqm}.

\begin{lemma}\label{lem:qm-vanishing}
	Let $\zeta$ be a partial contact quasi-state on $M$ with respect to $\alpha$. Denote by $\tau$ the induced contact quasi-measure. Let $A\subset M$ be the preimage of a closed subset under a contact $\alpha$-involutive map $F:M\to\R^N$. If $A$ is displaceable by ${\rm Cont}_0(M)$, then $\tau(A)=0.$
\end{lemma}
\begin{proof}
Since $A\subset M$ is closed, there exists an open neighbourhood $U\supset A$ that is displaceable by ${\rm Cont}_0(M).$ Let $\chi: M\to [0,1]$ be a smooth Reeb-invariant function with support in $U$ such that $\left.\chi\right|_{A}=1.$ Then,
\[ \tau(A)=\inf\left\{ \zeta(\chi\cdot h)\:|\: h:M\to[0,1]  \text{ Reeb invariant, smooth,  and }\left.h\right|_{A}=1\right\}. \]
Since ${\rm supp}\:(\chi \cdot h)$ is displaceable by ${\rm Cont}_0(M),$ the vanishing property of $\zeta$ implies $\tau(A)=0.$
\end{proof}

\section{Gapped modules}\label{sec-gap-mod}

In this section, we establish a general theory for persistence modules that supports those from (\ref{per-mod-CHD}). Recall that a persistence $\k$-module, in a general sense, is an $\mathcal I$-parametrized family of $\k$-modules denoted by $\mathbb V = \{V_i\}_{i \in \mathcal I}$, where the index set $\mathcal I$ is a subset of $\R$, together with a certain family of $\k$-linear maps. Explicitly, for each $i \leq j$ in $\mathcal I$, there is a $\k$-linear map $\iota_{i,j}: V_i \to V_j$ satisfying $\iota_{j,k} \circ \iota_{i,j} = \iota_{i,k}$ (and $\iota_{i,i} = \mathds{1}_{V_i}$). Suppose $\mathcal I$ is totally ordered, then by \cite{C-B-decomposition}, the persistence $\k$-module $\mathbb V$ admits a decomposition as follows, 
\begin{equation} \label{normal form}
\mathbb V = \bigoplus_{I \subset B(\mathbb V)} \k_I 
\end{equation}
where $B(\mathbb V)$, called the barcode of $\mathbb V$, is a collection of some intervals $I$ in $\mathcal I$. Here, an interval $I$ of $\mathcal I$ means that for any $i, k \in I$ and $j \in I$ with $i \leq j \leq k$, we have $j \in \mathcal I$. Moreover, $\k_I$ denotes the trivial rank-1 bundle over the interval $I$. For more details on persistence module theory, see \cite{CZ05,CC-SGGO-prox-09,CdeSGO16,PRSZ20}.

\subsection{Definition of gapped modules} Fix a scalar $\lambda \geq 0$ and a subset $\mathcal I \subset \R$. For $s, t \in \mathcal I$, denote by $s \leq_{\lambda} t$ the relation either $s =t$ or $s \leq t -\lambda$ (cf.~(\ref{po-CHD})). This relation defines a partial order on $\mathcal I$, which may not be a total order (since there is no relation between $s$ and $s + \varepsilon$ if $s, s+\varepsilon \in \mathcal I$ with $\varepsilon < \lambda$, for instance). 
 
\begin{dfn} \label{dfn-gap-mod} Fix a field $\k$ and a scalar $\lambda \geq 0$. A {\rm $\lambda$-gapped $\k$-module} $\mathbb V$ consists of the following data 
\[ \mathbb V= (\{V_t\}_{t \in \mathcal I}, \{\iota_{s,t}: V_s \to V_t\}_{s \leq_{\lambda} t}\}) \]
where $V_t$ is a finite dimensional $\k$-module, $\iota_{t,t} = \mathds{1}$ and $\iota_{s,t} \circ \iota_{r,s} = \iota_{r,t}$ for $r \leq_{\lambda} s \leq_{\lambda} t$. \end{dfn}

As mentioned above, a persistence $\k$-module can be defined over a partially order set, say $(\mathcal I, \leq_{\lambda})$, but the decomposition as in (\ref{normal form}) only works for a totally ordered parametrization set. Therefore, to extract information from $(\mathcal I, \leq_{\lambda})$ in terms of the barcode, we will restrict our parametrization set from $\mathcal I$ to certain discrete subsequences $\mathfrak a = \{\cdots, \eta_i, \eta_{i+1}, \cdots\} \subset \mathcal I$. In what follows, we will always view $\mathfrak a$ as an injective map $\mathfrak a: \Z \to \mathcal I$ such that $\eta_i = \mathfrak a(i)$.

\begin{dfn} \label{dfn-restriction} Fix a field $\k$, a scalar $\lambda > 0$, and a $\lambda$-gapped $\k$-module $\mathbb V$ parametrized by $\mathcal I$. A {\rm gapped restriction with respect to $\mathfrak a: \Z \to \mathcal I$}, denoted by $\mathbb V(\mathfrak a)$, is a persistence $\k$-module indexed by $\mathcal \Z$ (equipped with its total order) defined by
\[ \mathbb V(\mathfrak a)_i : = V_{\mathfrak a(i)} \]
where $\mathfrak a(i+1) - \mathfrak a(i) = \lambda$ for each $i \in \Z$. Moreover, the $\lambda$-gapped restriction is is called {\rm normalized} if $\mathfrak a(0) \in [0,\lambda)$.  
 \end{dfn}
 
 Here are some examples of gapped $\k$-modules and their gapped restrictions. 
 
 \begin{ex} Any $1$-gapped $\k$-module $\mathbb V$ over $\mathcal I = \R$ admits a normalized gapped restrictions by considering $\mathfrak a: \Z \to \mathcal I = \R$ defined by 
\[ \mathfrak a(i) = i.  \]
Therefore, the standard $\mathbb Z$-parametrized $\k$-persistence modules fit into Definition \ref{dfn-restriction}. Of course, one can consider other $\mathfrak a: \Z \to \mathcal I = \R$, for instance, $\mathfrak a(i) = i+\frac{1}{2}$, which also defines a normalized gapped restriction. \end{ex}

\begin{remark} Note that not every $\lambda$-gapped $\k$-module admits gapped restrictions. For instance, if a $1$-gapped module $\mathbb V$ is parametrized by $\mathcal I = \{0, 1, 2, 2^2, \cdots, 2^n, \cdots\}$, then it does {\rm not} admit any gapped restriction essentially because $\mathcal I$ is too sparse. \end{remark}

\begin{ex} \label{ex-chfh-gap} For any Liouville fillable contact manifold $(M, \xi = \ker \alpha)$ and a contact Hamiltonian $h: [0,1] \times M \to \R$, via the contact Hamiltonian Floer homology, one can construct an ${\rm osc}_Rh$-gapped $\k$-module over $\R \backslash \mathcal S_h$ as in Section \ref{sec:persistence}, denoted by $\mathbb{P}(W,h)$ or $\mathbb P(h)$, where $\mathcal S_h$ is the closed nowhere dense subset of $\R$ defined by (\ref{eq:spectrum}). Moreover, it admits a normalized gapped restriction. Indeed, due to the discreteness of $\mathcal S_h$, there exists 
\[ \eta_0 \in [0, {\rm osc}_Rh) \cap (\R \backslash \mathcal S_h) \]
such that $\eta_i = \eta_0 + i \cdot {\rm osc}_Rh \in \R \backslash \mathcal S_h$. Then, by setting $\mathfrak a(i) := \eta_i$ for all $i \in \mathbb Z$, we obtain a desired normalized gapped restriction with respect to $\mathfrak a$. \end{ex}
 
\noindent Here are some basic observations on gapped $\k$-modules. 

\medskip

\noindent (i) By the general theory on persistence modules as explained at the beginning of this section, for a $\lambda$-gapped $\k$-module, any gapped restriction $\mathbb V(\mathfrak a)$ admits a barcode $B(\mathbb V(\mathfrak a))$, where the endpoints of bars come from those $\eta_i = \mathfrak a(i)$. 

\medskip

\noindent (ii) Let $\mathbb V$ be a $\lambda$-gapped $\k$-modules parametrized by $\mathcal I.$ For two gapped restrictions $\mathbb V(\mathfrak a)$ and $\mathbb V(\mathfrak b)$, they are comparable in the sense that, by a shift of $|\mathfrak a(0) - \mathfrak b(0)|$, the two parameterization sets $\mathfrak a$ and $\mathfrak b$ coincide as maps from $\Z$ to $\mathcal I$. In particular, if $\mathfrak b$ is defined by a shift of $\mathfrak a$ by $n$ as follows, 
\[ \mathfrak b: \Z \to \mathcal I \,\,\,\,\mbox{by}\,\,\,\, i \mapsto \mathfrak b(i) := \mathfrak a (i + n),\]
then the shift will be $n \lambda$, and such a parametrization $\mathfrak b$ is denoted by $\mathfrak a[n]$. Following the standard notation of persistence $\k$-modules, we have $\mathbb V(\mathfrak a[n]) = \mathbb V(\mathfrak a) [n \lambda]$. This says that even though the images $\{\mathfrak a(i)\}_{i \in \Z}$ and $\{\mathfrak b(i)\}_{i \in \Z}$ are the same subsequence in $\mathcal I$ (as two sets), the corresponding gapped restrictions $\mathbb V(\mathfrak a)$ and $\mathbb V(\mathfrak b)$ may differ a lot in terms of the barcodes $B(\mathbb V(\mathfrak a))$ and $B(\mathbb V(\mathfrak b))$. 

\medskip

On the contrary, for two normalized gapped restrictions of $\lambda$-gapped $\k$-module $\mathbb V$, we have the following stability result.  
Recall that persistence $\k$-modules are quantitatively comparable via the interleaving distance $d_{\rm inter}$, defined as a certain shifted persistence isomorphism (see Section 1.3 in \cite{PRSZ20}). It essentially exploits the fact  that the parameterization set $\mathcal I \subset \R$ of a persistence module can be shifted.

\begin{prop} \label{prop-change-function}
Fix a scalar $\lambda > 0$ and a $\lambda$-gapped $\k$-module $\mathbb V$. Then, for any normalized gapped restrictions $\mathbb V(\mathfrak a)$ and $\mathbb V(\mathfrak b)$, we have $d_{\rm inter}(\mathbb V(\mathfrak a), \mathbb V(\mathfrak b)) \leq 2\lambda$. \end{prop}

\begin{proof} Set $\mathfrak a = \{\cdots, \eta_0, \eta_1, \cdots\}$ where $\mathfrak a(i) = \eta_i$ and $\mathfrak b = \{\cdots, \tau_0, \tau_1, \cdots\}$ where $\mathfrak b(i) = \tau_i$. By definition, without loss of generality, assume that $0 \leq \eta_0 \leq \tau_0 < \lambda$. Without loss of generality, let us consider $\ell \in \N$. According to Definition \ref{dfn-restriction}, we have
\begin{align*}
\eta_{\ell} = \eta_0 + \ell \lambda &\leq \tau_0 + (\ell+1) \lambda \\
& \leq (\tau_0 + (\ell+2) \lambda) - \lambda  = \tau_{\ell+2} - \lambda.
\end{align*}
Therefore, by Definition \ref{dfn-gap-mod}, there exists a morphism $\iota_{\eta_\ell \,\tau_{\ell+2}}: V_{\eta_\ell} \to V_{\tau_{\ell+2}}$. 
In a symmetric way, there exists a morphism $\iota_{\tau_{\ell} \,\eta_{\ell+2}}: V_{\tau_\ell} \to V_{\eta_{\ell+2}}$. In terms of notations, $V_{\eta_\ell} = \mathbb V(\mathfrak a)_{\ell}$ and $V_{\tau_{\ell+2}} = (\mathbb V(\mathfrak b)[2\lambda])_{\ell}$. Then we have the following commutative diagram,
\[\xymatrixcolsep{5pc}\ \xymatrix{ \mathbb V(\mathfrak a)_{\ell} \ar[r]^-{\iota_{\eta_{\ell} \, \tau_{\ell+2}}} \ar@/_2.5pc/[rr]^-{\iota_{\eta_{\ell} \,\eta_{\ell+4}}} & \left(\mathbb V(\mathfrak b)[2\lambda]\right)_{\ell} \ar[r]^-{\iota_{\tau_{\ell} \, \eta_{\ell+2}}[2\lambda]} & \left(\mathbb V(\mathfrak a)[4\lambda]\right)_{\ell}}\]
and a symmetric diagram. Then by definition we know $\mathbb V(\mathfrak a)$ and $\mathbb V(\mathfrak b)$ are $2\lambda$-interleaved, which yields the desired conclusion. 
\end{proof}

Here is a direct corollary of Proposition \ref{prop-change-function}. Recall that a quantitative comparison between barcodes is called the bottleneck distance denoted by $d_{\rm bottle}$, which transfers the distance $d_{\rm inter}$ into a combinatorial type distance, which is, in particular, easier to compute. For more details, see Section 2.2 in \cite{PRSZ20}. 

\begin{cor} \label{cor-long-bar} Fix a scalar $\lambda > 0$ and a $\lambda$-gapped module $\mathbb V$. Then any two gapped restrictions $\mathbb V(\mathfrak a)$ and $\mathbb V(\mathfrak b)$ satisfy 
\[ d_{\rm bottle}(B(\mathbb V(\mathfrak a)), B(\mathbb V(\mathfrak b))) \leq 2\lambda. \]
In particular, the cardinalities of infinite-length bars in the barcodes $B(\mathbb V(\mathfrak a))$ and $B(\mathbb V(\mathfrak b))$ are the same, and there is a one-to-one correspondence between those infinite-length bars with left endpoints shifted by at most $2\lambda$.  \end{cor}

\begin{proof} This directly follows from the standard isometry theorem in persistence module theory, say, the main result of \cite{BL15}, and the definition of $d_{\rm bottle}$. \end{proof}

Next, we consider two gapped modules with possibly two different gaps $\lambda, \lambda' \in \R$. One issue of comparing them is that these two gapped modules may be parametrized by different sets $\mathcal I$ and $\mathcal J$, so it is not obvious how to compare them via their gapped restrictions. To overcome this issue, we extend the gapped restriction (of a $\lambda$-gapped $\k$-module $\mathbb V$) in Definition \ref{dfn-restriction} to a {\it $\delta$-gapped restriction} (still of this $\mathbb V$) by simply relaxing the condition $\mathfrak a(i+1) - \mathfrak a(i) = \lambda$ to the following one:
\[ \mathfrak a(i+1) - \mathfrak a(i) = \delta \,\,\,\,\mbox{for any $i \in \Z$,} \]
for a given scalar $\delta>0$. Moreover, it is normalized if $\mathfrak a(0) \in [0, \delta)$. Then the following definition provides a possible approach to comparing two different gapped $\k$-modules. 

\begin{dfn} \label{dfn-gap-inter} Fix scalars $\lambda, \lambda' \in \R_{> 0}$. Let $\mathbb V$ be a $\lambda$-gapped $\k$-module parametrized by $\mathcal I$ and $\mathbb W$ be a $\lambda'$-gapped $\k$-module parametrized by $\mathcal J$, where $\mathcal I, \mathcal J$ are subsets of $\R$. These gapped $\k$-modules $\mathbb V$ and $\mathbb W$ are called {\rm $\delta$-interleaved} for a scalar $\delta \in \R_{\geq 0}$, either $\delta=0$ (only for $\mathbb V = \mathbb W$) or $\delta \geq \max\{\lambda, \lambda'\}$, if for any $\delta$-gapped restrictions $\mathbb V(\mathfrak a)$ and $\mathbb W(\mathfrak a)$ where $\mathfrak a = \{\cdots, \eta_0, \eta_1, \cdots\} \subset \mathcal I \cap \mathcal J$, we have morphisms 
\[ \varphi = \left\{\varphi_i:  \mathbb V(\mathfrak a)_{i} \to \left(\mathbb W(\mathfrak a)[\delta]\right)_{i} =\mathbb W(\mathfrak a)_{{i+1}}\right\}_{i \in \Z} \]
and 
\[ \psi = \left\{\psi_i:  \mathbb W(\mathfrak a)_{i} \to \left(\mathbb V(\mathfrak a)[\delta]\right)_{i} =\mathbb V(\mathfrak a)_{{i+1}}\right\}_{i \in \Z} \]
such that the following diagrams commute:
\[\xymatrixcolsep{5pc}\ \xymatrix{ \mathbb V(\mathfrak a)_{i} \ar[r]^-{\varphi_i} \ar@/_2.5pc/[rr]^-{\iota^{\mathbb V}_{i \,i+2}} & \left(\mathbb W(\mathfrak a)[\delta]\right)_{i} \ar[r]^-{\psi_i[\delta]} & \left(\mathbb V(\mathfrak a)[2\delta]\right)_{{i}}}\]
and 
\[\xymatrixcolsep{5pc}\ \xymatrix{ \mathbb W(\mathfrak a)_{i} \ar[r]^-{\psi_i} \ar@/_2.5pc/[rr]^-{\iota^{\mathbb W}_{i \,i+2}} & \left(\mathbb V(\mathfrak a)[\delta]\right)_{i} \ar[r]^-{\varphi_i[\delta]} & \left(\mathbb W(\mathfrak a)[2\delta]\right)_{i}}\]
for any $i \in \Z$. 
\end{dfn}

\begin{remark} The existence of a $\delta$-interleaving relation for some finite $\delta \geq 0$ between $\mathbb V$ and $\mathbb W$ automatically implies there exists a $\delta$-gapped sequence $\mathfrak a \subset \mathcal I \cap \mathcal J$. In general it is possible that $\mathcal I \cap \mathcal J = \emptyset$, but $\mathbb V$ and $\mathbb W$ are close to each other in terms of the interleaving relation. However, we will not encounter this situation in this paper. \end{remark} 

\begin{remark} We include $\delta=0$ in Definition \ref{dfn-gap-inter} to make sure that any $\lambda$-gapped $\k$-module $\mathbb V$ is $0$-interleaved to itself, where $\varphi = \psi$ and equal to the identity maps. Also, in Definition \ref{dfn-gap-inter} one can assume that $\mathfrak a$ is normalized, in the sense that $\mathfrak a(0) \in [0, \min\{\lambda, \lambda'\})$, since a shift of $\mathfrak a$ does not affect the diagrams in Definition \ref{dfn-gap-inter}.\end{remark}

\begin{remark} When $\lambda =0$, any $0$-gapped $\k$-module $\mathbb V$ parametrized by $\mathcal I$ reduces to a standard persistence $\k$-module. If furthermore, $\mathcal I = \R$, then any subsequence $\mathfrak a: \Z \to \mathcal I$ satisfying $\mathfrak a(i+1) - \mathfrak a(i) = \delta$ for any $i\in \Z$ induces a $\delta$-gapped restriction of $\mathbb V$, for any $\delta>0$. \end{remark}

\subsection{Algebraic spectral invariant} \label{ssec-alg-si}  In this section, we will explain how the contact spectral invariant defined in (\ref{def:spectral}) naturally fits into the framework of gapped modules. Recall that for a persistence module parametrized by $\mathcal I \subset \R$ and an element $a \in V_{\infty}: = \varinjlim_{i \to \infty} V_i$, one can define the spectral
invariant of class $a$ directly from the information of $B(\mathbb V)$. 

Recall that the set of infinite-length bars in $B(\mathbb V)$ correspond to a (not necessarily unique) basis of the $\k$-module $V_{\infty}$, denoted by $e = \{e_1, ... , e_n\}$ where $n = \dim_{\k} V_{\infty}$. Then under this basis $e$, the class $a$ admits a linear combination 
\begin{equation} \label{a-linear}
a = x_1 e_1 + \cdots x_n e_n, \,\,\,\,\mbox{where $x_i \in \k$.}
\end{equation}
Introduce the notation $s(e_i)$ as the left endpoint of the infinite-length bar corresponding to the basis element $e_i$, called the filtration spectrum (of $e_i$) as in Definition 5.2 in \cite{UZ16}. Then the spectral invariant of $a$, denoted by $c(a, \mathbb V)$, can be defined as follows (cf.~Proposition 6.6 in \cite{UZ16}):
\begin{equation} \label{def-si}
c(a, \mathbb V) := \max\{s(e_i) \in \R \cup \{-\infty\}\, | \, x_i \neq 0 \,\,\mbox{as in (\ref{a-linear})} \}
\end{equation}
where $c(a, \mathbb V) = -\infty$ precisely corresponds to $a = 0$ (where all $x_i = 0$).  A general result, say Theorem 7.1 in \cite{UZ16}, proves that the multi-set of the filtration spectra $\{s(e_i)\}_{i=1, ..., n}$ is unique for a fixed persistence $\k$-module $\mathbb V$ even though basis $e$ is not uniquely determined in general. This implies that the spectral invariant $c(a, \mathbb V)$ is well-defined, that is, independent of the choice of the basis $e$ above.

\medskip

Now, in the case of a $\lambda$-gapped $\k$-module $\mathbb V$, we will define the spectral invariant of $a$ with the help of gapped restrictions. 

\begin{dfn} \label{dfn-si-gap} Fix $\lambda \in \R_{>0}$ and let $\mathbb V$ be a $\lambda$-gapped $\k$-module parametrized by $\mathcal I$ where $+\infty$ is an accumulation point of $\mathcal I$. For any $a \in V_{\infty}$, define its spectral invariant by 
\[ c(a, \mathbb V): = - \inf \left\{c(a, \mathbb V(\mathfrak a))\,\bigg| \, \begin{array} {l} \mbox{$\mathbb V(\mathfrak a)$ is a normalized} \\ \mbox{$(\lambda$-$)$gapped restriction} \end{array} \right\}. \]
In particular, $c(0, \mathbb V) = +\infty$. 
\end{dfn}

\begin{remark} In Definition \ref{dfn-si-gap}, the condition that $+\infty$ is an accumulation point of $\mathcal I$ is to guarantee that $V_{\infty}$ is not empty, where classes $a$ can be taken from $V_{\infty}$. Also, we put the negative sign in the definition of $c(a, \mathbb V)$ in order to be consistent with the spectral invariant defined in (\ref{eq:spectral}). 
\end{remark}

\begin{remark} Definition \ref{dfn-si-gap} does not consider the case when $\lambda=0$ (hence, $0$-gapped $\k$-modules), since there are no $0$-gapped restrictions. However, one can carry out a limit argument based on $\delta$-gapped restriction (for $\mathcal I = \R$) for every $\delta>0$ and conclude that the definition $c(a, \mathbb V)$ in Definition \ref{dfn-si-gap} coincides with the computational definition in (\ref{def-si}). \end{remark}

\begin{remark} \label{rmk-generalized-csi} One can also define $c(a, \mathbb V)$ via gapped restrictions, not necessary normalized. Consider the following quantity, 
\[ \overline{c}(a, \mathbb V): = - \inf \left\{c(a, \mathbb V(\mathfrak a)) - \mathfrak a(0) \,\big| \, \mbox{$\mathbb V(\mathfrak a)$ is a $(\lambda$-$)$gapped restriction of $\mathbb V$}\right\}. \]
One can easily check that $c(a, \mathbb V) = \overline{c}(a, \mathbb V)$. Indeed, for any sequence $\mathfrak a: \Z \to \mathcal I$, the shifted-by-$\mathfrak a(0)$ gapped restriction $\mathbb V(\mathfrak a)[-\mathfrak a(0)]$ is equal to $\mathbb V(\mathfrak b)$ for $\mathfrak b: \Z \to \mathcal I$ defined by $\mathfrak b(i) := \mathfrak a(i) - \mathfrak a(0)$ (therefore, in particular, $\mathfrak b(0) = 0 \in [0, \lambda)$). This implies that is $\mathbb V(\mathfrak a)[-\mathfrak a(0)]$ is a normalized gapped restriction, where the resulting filtration spectrum is changed by $- \mathfrak a(0)$. \end{remark}

Here is an important property of $c(a,\mathbb V)$ related to the stability in Theorem \ref{thm:spectral-properties}. 

\begin{prop} \label{prop-alg-stability} Fix scalars $\lambda, \lambda' \in \R_{>0}$ and finite $\delta \geq \max\{\lambda, \lambda'\}$. Suppose $\mathbb V$ is a $\lambda$-gapped $\k$-module parametrized by $\mathcal I$ and $\mathbb W$ is a $\lambda'$-gapped $\k$-module parametrized by $\mathcal J$, where $\mathcal I, \mathcal J$ are dense subsets of $\R$. If $\mathbb V$ and $\mathbb W$ are $\delta$-interleaved, then for any $a \in V_{\infty} \cap W_{\infty}$, we have
\[ |c(a, \mathbb V) - c(a, \mathbb W)| \leq \delta. \]
If $\mathbb V = \mathbb W$, then $c(a, \mathbb V) = c(a, \mathbb W)$. 
\end{prop}

\begin{proof} By Definition \ref{dfn-si-gap}, for any $\varepsilon>0$, there exists a normalized gapped restriction $\mathbb V(\mathfrak a)$ of $\mathbb V$ such that $c(a, \mathbb V(\mathfrak a))  \leq -c(a, \mathbb V) + \varepsilon$. According to (\ref{def-si}), $c(a, \mathbb V(\mathfrak a)) = s(e_*)$, as the left endpoint of an infinite-length interval $I \in B(\mathbb V(\mathfrak a))$ corresponding to a basis element $e_*$ of $V_{\infty}$ (in particular, $s(e_*)=\mathfrak a(i)$ for some $i \in \Z$). Now, due to the denseness of $\mathcal I$ in $\R$, choose a normalized $\delta$-gapped restriction of $\mathbb V$, denoted by $\mathbb V(\mathfrak b)$ for some $\mathfrak b: \Z \to \mathcal I$, so that $c(a, \mathbb V(\mathfrak b)) = s(e_*) = c(a, \mathbb V(\mathfrak a))$. Moreover, due to the denseness of $\mathcal J$ in $\R$, up to small perturbations we can assume that $\mathfrak b \subset \mathcal J$ as well. In this way, we obtain normalized $\delta$-gapped restrictions $\mathbb V(\mathfrak b)$ and $\mathbb W(\mathfrak b)$. By our assumption and Definition \ref{dfn-gap-inter}, persistence $\k$-modules $\mathbb V(\mathfrak b)$ and $\mathbb W(\mathfrak b)$ are $\delta$-interleaved. This yields the following comparison:
\[ |c(a, \mathbb V(\mathfrak b)) - c(a, \mathbb W(\mathfrak b))| \leq \delta, \]
by a standard argument. Then we have 
\begin{align*}
c(a, \mathbb V) - \delta & \leq - c(a, \mathbb V(\mathfrak a)) + \varepsilon - \delta \\
&= -c(a, \mathbb V(\mathfrak b)) - \delta  + \varepsilon\\
 & \leq - c(a, \mathbb W(\mathfrak b)) + \varepsilon  \leq  c(a, \mathbb W) + \varepsilon.
\end{align*}
Thus we obtain the desired conclusion by switching $\mathbb V$ and $\mathbb W$ and let $\varepsilon \to 0$. 
\end{proof}

The following example shows how the contact spectral invariant $c(W, h, \theta)$ defined in  (\ref{def:spectral}) and derived from contact Hamiltonian dynamics naturally fits into the algebraic framework developed in this section.

\begin{ex}\label{ex-alg-csi} For a contact Hamiltonian $h:[0,1]\times M\to\R_+$ and a Liouville filling $W$ of $M$, as explained in Example \ref{ex-chfh-gap} one obtains an ${\rm osc}_Rh$-gapped $\k$-module $\mathbb P(W, h)$, where $\k$ is the ground field. Similarly, one obtains an ${\rm osc}_Rg$-gapped $\k$-module $\mathbb P(W, g)$ for contact Hamiltonian $g: [0,1] \times M \to \R_+$. For both $\mathbb P(W, h)$ and $\mathbb P(W, g)$, we have 
\[ \mathbb P(W, h)_{\infty} = \mathbb P(W, g)_{\infty} = {\rm SH}_*(W). \]
Meanwhile, by Section 6 in \cite{UZ16} which relates the filtration spectrum of the basis of $\mathbb P(W, h)_{\infty}$ with the persistence module structure, for any class $\theta \in {\rm SH}_*(W)$ we have
\begin{equation} \label{c-c}
c(\theta,  \mathbb P(W, h)) = c(W, h, \theta) ( = c(h, \theta))
\end{equation}
where the left-hand side is from (\ref{dfn-si-gap}) and right-hand side is from (\ref{def:spectral}). Moreover, $\mathbb P(W, h)$ is parametrized by $\mathcal I = \R \backslash \mathcal S_h$ and $\mathbb P(W, g)$ is parametrized by $\mathcal J = \R \backslash \mathcal S_g$, where both $\mathcal I$ and $\mathcal J$ are dense in $\R$. Following the notations in the hypothesis of Proposition \ref{prop-alg-stability}, $\lambda = {\rm osc}_Rh$, $\lambda' = {\rm osc}_Rg$, and one can prove, in a similar way as the proof of (4) in Theorem \ref{thm:spectral-properties}, two gapped $\k$-modules $\mathbb P(W, h)$ and $\mathbb P(W, g)$ are $C|{\rm osc}_R(h,g)|$-interleaved. Here, we choose $C$ sufficiently large so that  
\[ C|{\rm osc}_R(h,g)| \geq \max\{{\rm osc}_Rh, {\rm osc}_Rg\}. \]
Then Proposition \ref{prop-alg-stability} applies, which yields a stability result on $c(h, \cdot)$ (cf.~(4) in Theorem \ref{thm:spectral-properties}). 
\end{ex}

 Here are two drawbacks when the gapped $k$-module theory is realized in a geometric setting, as in Example \ref{ex-alg-csi} above. One, the upper bound shown up in the stability for the contact spectral invariant $c(h, \cdot)$, which is $C|{\rm osc}_R(h,g)|$, is not always sharp.

Second, only the setting of {\rm Liouville} fillable contact manifolds has been discussed in this section so far (in particular, the equality in (\ref{c-c})). For a general weakly\textsuperscript{+} monotone symplectic filling as in the hypothesis of Definition \ref{def:spectral}, extra complexity arises from the fact that the contact Floer homology group ${\rm HF}_*(\eta \#h)$ is finite-dimensional over a certain Novikov field $\Lambda$ (based on the ground field $\k$). Persistence module theory over a Novikov field $\Lambda$ has been developed, say in \cite{UZ16}, but on the filtered chain level. Therefore, in order to establish a corresponding algebraic framework, one needs to set up a ``chain model'' for ${\rm HF}_*(\eta\#h)$, that is, a filtered chain complex (over $\Lambda$) denoted by $(C_*(W, h), \partial_*, \ell)$ so that its filtered homologies satisfy the following relation:
\[ H_*^{<\eta}(W, h) : = \frac{\ker(\partial_*: C^{\{\ell< \eta\}}_*(W, h) \to C^{\{\ell< \eta\}}_{*-1}(W, h))}{{\rm im}(\partial_{*+1} : C^{\{\ell< \eta\}}_{*+1}(W, h) \to C^{\{\ell< \eta\}}_{*}(W, h))} \simeq {\rm HF}_*(\eta\#h).\]  
Conjecturally\footnote{This was discussed with Jungsoo Kang and Beomjun Sohn in 2024.}, such a filtered chain complex is the one that supports a twisted version of the classical Rabinowitz Rloer homology of the filling $W$. This will be explored somewhere else.

\subsection{Duality} Persistence module theory admits an algebraic duality, on different levels --- filtered chain complex (see Section 2.4 in \cite{UZ16}), barcode (see Section 6.2 in \cite{UZ16}), persistence module (see Section 6.3 in \cite{BCZ-K}). For gapped $\k$-modules discussed in this section, its algebraic duality will be a word-by-word rephrasing of its main ingredients, starting from reversing the partial order $\leq_{\lambda}$. Inspired by this algebraic duality, in what follows we will elaborate on the duality that naturally appears in our contact geometric setting.

Recall the Poincar\'e duality on symplectic (co)homology. In terms of the notation from \cite{CO}, we have the following duality, 
\begin{equation} \label{pd-sh}
{\rm PD}: {\rm SH}_*(W) \simeq {\rm SH}^{-*}(W). 
\end{equation}
The symplectic cohomology ${\rm SH}^{-*}(W)$ is defined as the inverse limit of Hamiltonian Floer {\it homologies} ${\rm HF}_*(H^a)$ on the completion $\widehat{W}$, where $H(r,x) = ra + C$ on the cylindrical end of the completion $\widehat{W}$ and a {\it negative} constant function $a <0$. Therefore, we have 
\begin{equation} \label{HF-inverse}
\varprojlim_{\eta} {\rm HF}_*(\eta \#h) = {\rm SH}^{-*}(W) 
\end{equation}
where the inverse limit is taken over $\eta \in \R$ when $\eta \to -\infty$. Therefore, following Definition \ref{def:spectral}, one can define contact spectral invariant for classes in ${\rm SH}^{-*}(W)$. 

On the level of chain complexes that define the symplectic (co)homologies, the duality in (\ref{pd-sh}) is induced by the following isomorphism,  
\begin{equation} \label{chain-iso}
{\rm CF}_*(H, J) \simeq {\rm CF}^{-*}(\bar{H}, \bar{J})
\end{equation}
where ${\rm CF}^{-*}(\bar{H}, \bar{J})$ can be algebraically identified with ${\rm Hom}({\rm CF}_{-*}(\bar{H}, \bar{J}); \k)$. Here, $\bar{H}(t,x) := -H(-t, x)$ and $\bar{J}_t := J_{-t}$. More explicitly, the isomorphism in (\ref{chain-iso}) is the identification of closed Hamiltonian orbits by reversing the time. This implies that the associated ${\rm osc}_Rh$-gapped module from (\ref{HF-inverse}) can be denoted by $\overline{\mathbb P}(W, \bar{h})$, where $\bar{h}(t,x) = -h(-t, x)$. Here $\overline{\mathbb P}(\cdot)$ means the parametrization set $\R \backslash \mathcal S_h$ is negated (that is, $s \to - s$), so the partial order $\leq_{{\rm osc}_Rh}$ is reversed accordingly. Indeed, one verifies that for each $\eta \in \R$ and any point $(t,x) \in [0,1] \times M$, we have 
\begin{align*}
\left((-\eta) \# \bar{h}\right)(t,x) & = (-\eta) + \bar{h}(t, \phi_R^{\eta t}(x)) \\
& = -\eta - h(-t, \phi_R^{\eta t}(x)) \\
& = - \left(\eta + h\circ \phi_R^{-\eta t}\right)(-t,x) = \left(\overline{\eta \#h}\right) (t,x).
\end{align*}
This implies that 
\begin{equation}\label{hf-iso}
{\rm HF}_*(\eta \#h)\simeq {\rm HF}^{-*}(\overline{\eta \#h}) \simeq {\rm HF}^{-*}((-\eta) \# \bar{h}). 
\end{equation}
The first isomorphism in (\ref{hf-iso}) is given by the Poincar\'e duality as in (\ref{chain-iso}). 

\begin{remark} Note that, following the convention in Section 3.1 in \cite{CO}, the algebraic duality does {\it not} change the sign of the filtration since on the cochain complexes one considers super-level sets instead of sublevel sets (as a comparison, following the convention in Section 6.2 in \cite{UZ16}, the algebraic duality changes the sign of the filtration). All in all, here only the Poincar\'e duality in (\ref{chain-iso}) changes the sign of the filtration.\end{remark}

The following relates the contact spectral invariant between the duals. 

\begin{prop}[Duality]\label{prop-duality}  Let $(M, \xi)$ be a contact manifold with a Liouville filling $(W, \lambda)$. Then, for any contact Hamiltonian $h: [0,1] \times M \to \R_+$ and a class $\theta \in {\rm SH}_*(W)$, we have 
\[ c(h, \theta) = - c(\bar{h}, {\rm PD}(\theta)) \]
where ${\rm PD}$ is the Poincar\'e duality in (\ref{pd-sh}) and $\bar{h}(t,x) = -h(-t, x)$. 
\end{prop}

\begin{proof} Identify $c(h, \theta)$ with the algebraic spectral invariant $c(\theta, \mathbb{P}(W,h))$ or simply $c(\theta, \mathbb{P}(h))$, and identify $c(\bar{h}, {\rm PD}(\theta))$ with the algebraic spectral invariant $c({\rm PD}(\theta), \overline{\mathbb{P}}(\bar{h}))$.  By Definition \ref{dfn-si-gap}, up to $\varepsilon>0$, suppose $-c(\theta, \mathbb{P}(h))$ is obtained via a normalized ${\rm osc}_Rh$-gapped restriction, denoted by $\mathbb{P}(h; \mathfrak a)$, for some index sequence $\mathfrak a: \Z \to \R \backslash \mathcal S_h$. Consider a new sequence $\mathfrak b: \Z \to \R\backslash \mathcal S_h$ (note that $\mathcal S_h = \mathcal S_{\bar{h}}$) defined by 
\[ \mathfrak b(i)  = - \mathfrak a(i). \]
Then the resulting persistence module is a normalized ${\rm osc}_R\bar{h}$-gapped restriction of $\mathbb{P}(\bar{h})$. Moreover, two persistence modules $\mathbb{P}(h; \mathfrak a)$ and $\overline{\mathbb{P}}(\bar{h}; \mathfrak b)$ are dual to each other. Meanwhile, if the class $\theta \in {\rm SH}_*(W)$ is generated by basis elements $e_1, ..., e_m$, then ${\rm PD}(\theta)$ is generated by ${\rm PD}(e_1), ..., {\rm PD}(e_m)$. Since ${\rm PD}$ changes the sign of the filtration, one gets $c({\rm PD}(\theta), \overline{\mathbb P}(\bar{h}; \mathfrak b)) = -c(\theta, \mathbb P(h; \mathfrak a))$. Therefore, 
\[ - c({\rm PD}(\theta), \overline{\mathbb P}(\bar{h})) \leq c({\rm PD}(\theta), \overline{\mathbb P}(\bar{h}; \mathfrak b)) = -c(\theta, \mathbb P(h; \mathfrak a)) = c(\theta, \mathbb P(h)). \]
Since $a = {\rm PD}({\rm PD}(a))$, a symmetric argument yields the desired conclusion. \end{proof}

To end this section, we mention that many algebraic invariants derived from persistence module theory can be adapted into gapped module theory, in particular, for homological invisible information such as (generalized) boundary depth. Further development in this direction will be carried out in the future work.

\section{Pair-of-pants product} \label{sec-pp}

In this section, we prove a maximum principle that enables us to define the pair-of-pants product ${\rm HF}_\ast(h^{a_1})\otimes {\rm HF}_\ast(h^{a_2})\to {\rm HF}_\ast(h^b)$ for contact Hamiltonian Floer homology. Here, the superscripts $a_1, a_2, b$ are used to label the input contact Hamiltonian functions, with no actual numerical meanings. This product is associated with a triple of admissible contact Hamiltonians $h^{a_1}, h^{a_2},$ and $h^b$ such that the following holds:
\begin{enumerate}
    \item there exists $\delta>0$ such that $h^{a_j}_t=0$ for $t\in  (-\delta,0] \cup (1-\delta, 1]$ and $j\in\{1,2\}.$
    \item there exists $\delta>0$ such that $h^{b}_t=0$ for $t\in  [0, \delta) \cup (\frac{1}{2}-\delta, \frac{1}{2}+\delta)\cup (1-\delta, 1]$.
    \item $h^{a_1} \bullet h^{a_2} \leqslant h^b$ where $h^{a_1} \bullet  h^{a_2} := \left\{  \begin{matrix} 2h^{a_1}_{2t}, & t \in \left[0, \frac{1}{2}\right] \\  2h^{a_2}_{2t-1}, & t \in \left[\frac{1}{2}, 1 \right]       \end{matrix}   \right..$
\end{enumerate}
The pair-of-pants product is defined by counting solutions of the Floer equation parametrized by a sphere with three points removed (i.e., by a pair-of-pants).  The precise definition of the pair-of-pants product is given in Section~\ref{sec:defPPP}. In Sections \ref{sec:pants} and \ref{sec:productdata}, we introduce the necessary objects for the definition of the pair-of-pants product, namely the pair-of-pants with a slit and product data. The Section~\ref{sec:compactnessPPP} proves the maximum principle (as a theorem by its own) that implies the compactness of the moduli spaces and thus justifies the definition. While we focus on the pair-of-pants product in this paper, our methods can be adapted to higher product operations, corresponding to genus-0 Reimannian surfaces with several negative and several positive ends.

\subsection{Pair-of-pants with a slit}\label{sec:pants}
Now we introduce the notation that is used in the definition of the product. This model was considered first in \cite[Section 3.2]{AS-gt}. Let $(P,j)$ be the Riemannian surface obtained by removing three points, $a_1$, $a_2$, $b$, from the sphere. Let $S\subset P$ be the subset of $P$ consisting of two intersecting curves as in Figure~\ref{fig_pp} such that the complement of $S$ is biholomorphic to $\R\times(0,1)\sqcup \R\times(0,1)$. 
\begin{figure}[h]
\includegraphics[scale=0.75]{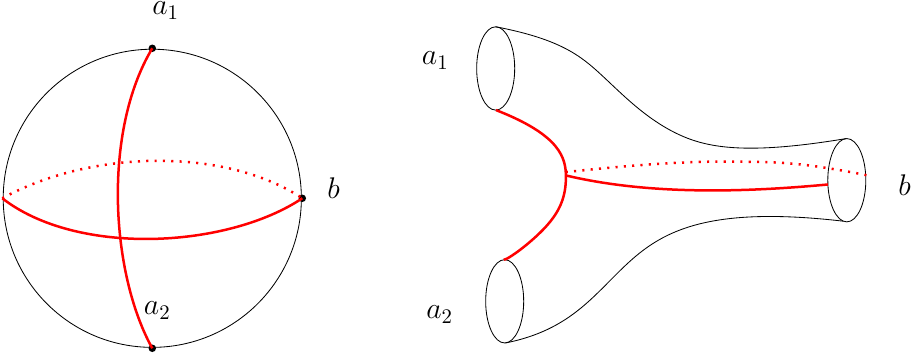}
\centering
\caption{Pair-of-pants with a slit.} \label{fig_pp}
\end{figure}
Let $\psi_1, \psi_2: \R\times(0,1)\to P\setminus S$ be restrictions of a biholomorphism $\R\times(0,1)\sqcup \R\times(0,1)\to P\setminus S$ to the connected components of $\R\times(0,1)\sqcup \R\times(0,1)$. Assume, by modifying the biholomorphism if necessary,
\begin{align*}
& \lim_{s\to -\infty} \psi_1(s,t)=a_1, \quad  \lim_{s\to +\infty} \psi_1(s,t)=b,\\
& \lim_{s\to -\infty} \psi_2(s,t)=a_2, \quad  \lim_{s\to +\infty} \psi_2(s,t)=b.
\end{align*} 
The limits above are limits in the sphere. Let $\delta>0$ be a sufficiently small positive number (say, smaller than $\frac{1}{2}$) and let $\Omega=\Omega_\delta\subset P$ be the open neighbourhood of $S$ given by
\[\Omega:= \left(P\setminus \psi_1( \R\times [\delta, 1-\delta] )\right)\setminus \psi_2( \R\times [\delta, 1-\delta] ).\]
Let 
\begin{align*}
&\iota_{a_1}, \iota_{a_2}: (-\infty, 0]\times\mathbb{S}^1\to P,\\
& \iota_b: [0,+\infty)\times\mathbb{S}^1\to P
\end{align*}
be biholomorphic embeddings such that 
\begin{align*}
& \iota_{a_k}(s,t)= \psi_k(s-s_a, t)\quad \text{for } s \in (-\infty, 0],  t \in (0, 1),\text{ and } k=1,2,\\
& \iota_{b}\left(\frac{s}{2},\frac{t}{2}\right)= \psi_1\left(s+s_b, t\right)\quad \text{for } s \in[0,+\infty) \text{ and } t \in (0, 1),\\
& \iota_{b}\left(\frac{s}{2},\frac{t+1}{2}\right)= \psi_2\left(s+s_b, t\right)\quad \text{for } s \in[0,+\infty) \text{ and } t \in (0, 1),
\end{align*}
for positive $s_a, s_b\in\R_{>0}$. Let $\beta$ be a $1-$form on $P$ such that 
\begin{align*}
&\psi^{*}_k \beta = dt \quad \text{on }\mathbb{R} \times [\delta, 1-\delta]\text{ for }k=0,1,\\
& \iota_{a_k}^\ast\beta = dt \quad \text{on } (-\infty, -2 s_a)\times\mathbb{S}^1 \text{ for } k=0,1,\\
& \iota^\ast_b\beta= 2 dt \quad \text{on } (2 s_b, +\infty)\times \mathbb{S}^1.
\end{align*}
Finally, let $d{\rm vol}_P$ be an area form on $P$ such that $\iota_{a_1}^\ast d{\rm vol}_P, \iota_{a_2}^\ast d{\rm vol}_P,$ and $\iota_{b}^\ast d{\rm vol}_P$ are equal to $ds\wedge dt$ for $\abs{s}$ sufficiently large,
\subsection{Product data}\label{sec:productdata}
Now, we define \emph{product data}. Let $\mathcal{D}^{a_1}:=(H^{a_1}, J^{a_1})$, $\mathcal{D}^{a_2}:=(H^{a_2}, J^{a_2})$, and $\mathcal{D}^b:=(H^{b}, J^{b})$ be regular Floer data such that the Hamiltonians $H^{a_1}_t$ and $H^{a_2}_t$ are constant for $t\in [0, \delta] \cup [1-\delta, 1] $, and such that $H^b_t$ is constant for $t\in \left[0, \frac{\delta}{2}\right]\cup\left[\frac{1-\delta}{2}, \frac{1+\delta}{2}\right]\cup \left[1-\frac{\delta}{2}, 1\right]$. The product data for the triple $(\mathcal{D}^{a_1}, \mathcal{D}^{a_2};\mathcal{D}^{b})$ consists of a (smooth) $P$-dependent Hamiltonian
\[H:\hat{W}\times P\to \R : (x,p)\mapsto H_p(x)\]
and of a (smooth) $P$-family $J_p$ of $\omega$-compatible almost complex structures on $\hat{W}$ such that the following conditions hold.
\begin{enumerate}
\item (Conditions on the ends of $P$).
\begin{align*}
& H(x, \iota_{a_1}(s,t))=H^{a_1}_t(x)\quad\text{for } s\leqslant -1,\\
& H(x, \iota_{a_2}(s,t))=H^{a_2}_t(x)\quad\text{for } s\leqslant -1,\\
& H(x, \iota_{b}(s,t))=H^{b}_t(x)\quad\text{for } s\geqslant 1,\\
&J_p=\left\{ \begin{matrix} J^{a_1}_t& \text{if } p\in\iota_{a_1}((-\infty,-1]\times\mathbb{S}^1), \\ J^{a_2}_t & \text{if } p\in\iota_{a_2}((-\infty,-1]\times\mathbb{S}^1),\\ J^b_t & \text{if } p\in\iota_{b}([1,+\infty)\times\mathbb{S}^1).\end{matrix}\right.
\end{align*}
\item (The Hamiltonian on the conical end of $\hat{W}$). There exist $T\in\R_{>0}$ and a smooth function $h: \partial W\to\R$ such that 
\[H_p(y,r)= r\cdot h(y,p) \]
 for $r\geqslant T$, $y\in\partial W$, and $p\in P$.
\item (The almost complex structure on the conical end of $\hat{W}$).  There exists $T\in\R_{>0}$ such that $J_p$ is of twisted SFT-type on $\partial W\times[T,+\infty)$ for all $p\in P$.
\item (Conditions around the slit). For all $p\in\Omega$, the contact Hamiltonian $h_p:= h(\cdot, p)$ is equal to 0.
\item (Monotonicity). The functions 
\[s\mapsto h(y, \psi_1(s,t)),\quad s\mapsto h(y, \psi_2(s,t))\]
are increasing for all $t\in (0,1)$ and $y\in \partial W$.
\end{enumerate}

The definition of product data uses the notion of almost complex structure of twisted SFT-type. We now recall this notion. Let $(M, \xi)$ be a contact manifold with a contact form $\alpha$. Let $q:M\to \R_{>0}$ be a smooth positive function. Denote by $Q:M\times \R_{>0}\to \R_{>0}$ the corresponding homogeneous Hamiltonian on the symplectization, i.e. $Q(y, r)= r\cdot q(y)$. Let $N_Q$ be the following distribution in $T(M\times\R_{>0})$
\[ N_Q(y,r):= \left\{\left(v, -\frac{r dq(y) v}{q(y)}\right)\:|\: v\in\xi_{y} \right\}. \]
Note that $dQ$ vanishes on $N_Q$. For $v\in \xi_y$, denote 
\[\zeta_Q^r(v):= \left(v, -\frac{r dq(y) v}{q(y)}\right)\in N_Q(y,r). \]
An almost complex structure $J$ on $M\times\R_{>0}$ is said to be of twisted SFT-type (with the twist $Q$) if there exists a $d\alpha$-compatible complex structure $j$ on $\xi$ such that the following holds
\begin{itemize}
\item $ J X_Q= r\partial_r, $
\item $ J\zeta^r_Q(v)= \zeta^r_Q(jv). $
\end{itemize}
If $J$ is of twisted SFT-type with the twist $Q$, then $dQ\circ J=-\lambda$.
\begin{remark}
The conditions on the 1-form $\beta$ and on the $P$-dependent Hamiltonian $H$ (specifically, conditions (4) and (5)) imply $(d_PH)\wedge \beta + H d\beta\geqslant 0$ on $(\partial W)\times\R_{\geqslant r}$ for $r$ sufficiently big. This is a key inequality used in the proof of Proposition~\ref{prop:ineq} (see \eqref{eq:importantineq} on page~\pageref{eq:importantineq}). Here, $d_P$ denotes the exterior derivative ``with respect to'' $P$. 
\end{remark}

\subsection{Definition of the pair-of-pants product}\label{sec:defPPP}

In this section, we outline the definition of the pair-of-pants product for contact Hamiltonian Floer homology. For the sake of simplicity, we give details for the case of Liouville domains and $\mathbb{Z}_2$ coefficients. The general case involves Novikov rings and is analogous to the definition of the pair-of-pants product in the closed case (see, for instance, \cite{piunikhin1996symplectic}).

Let $h^{a_1}, h^{a_2},$ and $h^{b}$ be admissible contact Hamiltonians as at the beginning of Section \ref{sec-pp}. Let $\mathcal{D}^{a_1}:=(H^{a_1}, J^{a_1})$, $\mathcal{D}^{a_2}:=(H^{a_2}, J^{a_2})$, and $\mathcal{D}^b:=(H^{b}, J^{b})$ be regular Floer data such that the Hamiltonians $H^{a_1}, H^{a_2}, H^b$ have slopes equal to $h^{a_1}, h^{a_2}, h^b,$ respectively, and such that they satisfy the following condition: the Hamiltonians $H^{a_1}_t$ and $H^{a_2}_t$ are constant for $t\in [0, \delta]\cup[1-\delta, 1] $, and the Hamiltonian $H^b_t$ is constant for $t\in \left[0, \frac{\delta}{2}\right]\cup\left[\frac{1-\delta}{2}, \frac{1+\delta}{2}\right]\cup [1-\frac{\delta}{2}, 1]$ for sufficiently small $\delta>0$. Let $(H,J)$ be product data for the triple $(\mathcal{D}^{a_1}, \mathcal{D}^{a_2};\mathcal{D}^{b})$.
The pair-of-pants product
\[\ast: {\rm HF}_\ast(h^{a_1})\otimes {\rm HF}_\ast(h^{a_2})\to {\rm HF}_\ast(h^b)\]
is defined, on the chain level, on generators, by
\[\gamma_1\ast \gamma_2:= \sum_{\gamma_3} n(\gamma_1, \gamma_2;\gamma_3) \left<\gamma_3\right>,\]
where $n(\gamma_1,\gamma_2;\gamma_3)$ is the number modulo 2 of the isolated solutions $u:P\to\hat{W}$ of the problem\footnote{Recall $(du-X_H(u)\otimes\beta)^{0,1}:= \frac{1}{2}\left( (du-X_H(u)\otimes\beta) + J(u)(du-X_H(u)\otimes\beta)\circ j \right).$}
\begin{align*}
    &(du-X_H\otimes\beta)^{0,1}=0,\\
    & \lim_{s\to-\infty} u\circ\iota_{a_k}(s,t)=\gamma_1(t),\text{ for } k=1,2,\\
    & \lim_{s\to+\infty} u\circ\iota_{b}(s,t)=\gamma_3(t).
\end{align*}

\subsection{Compactness of the moduli space}\label{sec:compactnessPPP}

Now, we prove that the elements of the moduli space cannot ``escape'' to the conical end of $\hat{W}$. Namely, we prove the following result, which might be of interest in its own right.
\begin{theorem}\label{thm:noescape}
    Let $P, j, \beta, d{\rm vol}_P$ be as in Section~\ref{sec:pants}. Let $W$ be a weakly\textsuperscript{+} monotone strong filling of a closed contact manifold. Let $E\in\R_{>0}$ and let $(H,J)$ be product data on the completion $\hat{W}$. Then, there exists a compact $K = K_{H,J,E}\subset \hat{W}$ such that $u(P)\subset K$ for every solution $u: P\to\hat{W}$ of the Floer equation
    \[\left(du - X_H(u)\otimes \beta\right)^{0,1}=0\]
    that satisfies $\mathbb{E}(u)=\frac{1}{2}\int_P  \abs{du - X_{H_p}\otimes\beta}^2  d{\rm vol}_P< E.$
\end{theorem}
To this end, we use the Aleksandrov maximum principle (see \cite[Theorem 9.1]{GT01}, \cite[Appendix A]{AS09}, and  \cite{MU}). The following proposition is a coordinate-free reformulation of the Aleksandrov maximum principle. 

\begin{prop}\label{prop:aleksandrov}
    Let $(\Sigma, j)$ be an open planar Riemannian surface (i.e. $(\Sigma, j)$ is a planar Riemannian surface that is not the sphere) with a volume form $dvol_\Sigma$. Let $K\subset \Sigma$ be a compact subset and let $A\in\R_{>0}$. Then, there exists a constant $C_{K,A}\in\R_{>0}$ such that for every connected open subset $\Omega\subset \Sigma$ with $\Omega\subset K$ and every 1-form $\eta$ on $\Omega$ with\footnote{Here, $\norm{\eta}_{L^2}= \int_\Omega\abs{\eta}^2 d{\rm vol}_\Sigma$, where $\abs{\eta}^2= \abs{\eta(v)}^2 +\abs{\eta(jv)}^2$ for some unit vector $v\in T\Sigma$. Exercise~2.2.3 in \cite{mcduff2025j} shows that $\abs{\eta}$ does not depend on the choice of $v.$} $\norm{\eta}_{L^2}\leqslant A$, the following holds. If $\mu, f:\Omega\to \R$ are smooth functions such that\footnote{Recall $d^C\mu:= d\mu\circ j.$}
    \begin{equation}\label{eq:ineqddc} -dd^C\mu + \eta\wedge d\mu \geqslant  f d{\rm vol}_\Sigma,\end{equation}
    then
    \[ \sup_{\Omega}\mu \leqslant \sup_{\partial \Omega} \mu^+ + C_{K,A}\cdot \int_\Omega f^2 d{\rm vol}_\Sigma.\]
    Here, $\mu^+:=\max(\mu, 0)$.
\end{prop}

By definition, for every product data $(H, J)$ there exists $T=T_{H,J}\in\R_{>0}$ such that $H(y,r,p)=r\cdot h(y,p)$ for $(y,r)\in (\partial W)\times[T, +\infty)$ and for some ($P$-parametrized) contact Hamiltonian $h:(\partial W)\times P\to\R$, and such that the restriction of $J_p$ to $(\partial W)\times[T, +\infty)$ is of twisted SFT-type. We denote by $Q^J:(\partial W)\times\R_{>0}\times P\to \R_{>0}$ the function such that $Q^J(\cdot,\cdot, p)$ is the twist of $J_p$. The next proposition shows that the function $p\mapsto \log(Q^J(u(p),p))$ satisfies the inequality of the type \eqref{eq:ineqddc}.

\begin{prop}\label{prop:ineq}
Let $(H,J)$ be product data on $\hat{W}$. Let $T=T_{H,J}$ and $Q=Q^J$. Let $u:P\to \hat{W}$ be a solution of the Floer equation
$0= \left( du - X_H(u)\otimes\beta\right)^{0,1}.$
Let $\Omega:=u^{-1}(\partial W\times(T, +\infty)).$ Then, the function $\mu:\Omega\to \R : p\mapsto \log Q(u(p), p)$ satisfies the inequality
\[ -dd^C\mu + \theta\wedge d\mu + d\theta\geqslant 0, \]
where
\[\theta:=\frac{(d_PQ\circ j)(u) + \{Q, H\}(u)\beta\circ j}{Q\circ u}.\]
\end{prop}

In Proposition~\ref{prop:ineq}, and in the rest of the paper, we denote by $d_P$ the exterior derivative ``with respect to $P$.'' In local coordinates $s+i t\in P$, we have $d_PQ= \partial_s Q ds + \partial_tQ dt.$ In particular, no derivatives with respect to $y\in\partial W$ and $r\in\R_{>0}$ are involved. Similarly $d_W$ denotes the exterior derivative ``with respect to $W.$''
\begin{proof}[Proof of Proposition~\ref{prop:ineq}]
\sloppy Let $\rho:\Omega\to \R_{>0}$ be the function given by $\rho(p):=Q_p(u(p))= e^{\mu(p)}$. We start the proof by computing $dd^C\rho=d(d\rho\circ j)$. The Floer equation, together with $d_WQ\circ J=-\lambda$, implies
\begin{align*}
d\rho = & d (Q\circ u)\\
	= & (d_P Q)\circ u + (d_W Q)(du)\\ 
	= & (d_P Q)\circ u + (d_W Q)(X_{H}(u)\beta - J(u)du\circ j + J(u)X_H(u)\beta\circ j)\\ 
	= &  (d_P Q)\circ u + \{Q, H\}(u)\beta + (d_W Q)J(u)( - du\circ j + X_H(u)\beta\circ j )\\
	= &  (d_P Q)\circ u + \{Q, H\}(u)\beta - \lambda ( - du\circ j + X_H(u)\beta\circ j )\\
	= &  (d_P Q)\circ u + \{Q, H\}(u)\beta + \lambda ( du\circ j ) - \lambda(X_H)(u)\beta\circ j.
\end{align*}
Therefore,
\begin{align*}
d\rho\circ j = & (d_P Q\circ j)\circ u + \{Q, H\}(u)\beta\circ j - \lambda ( du )+ \lambda(X_H)(u)\beta\\
		= &  (d_P Q\circ j)\circ u + \{Q, H\}(u)\beta\circ j - u^\ast \lambda+ \lambda(X_H)(u)\beta,
\end{align*}
and consequently,
\begin{align*}
dd^C \rho = & -u^\ast\omega + d (\lambda(X_H)(u))\wedge\beta +\lambda(X_H(u)) d\beta + d(\rho\theta)\\
		= & -u^\ast\omega - d (H(u))\wedge\beta -H(u) d\beta + d\rho\wedge\theta+ \rho d\theta.
\end{align*}
In the last equation, we also used $\lambda(X_H)=-H$. Since
\[ u^\ast \omega + d_WH(du)\wedge\beta = \frac{1}{2}\norm{du- X_H\otimes\beta}^2d{\rm vol}_P, \]
we have
\begin{align*}  
dd^C\rho =& -\frac{1}{2}\norm{du - X_H(u)\otimes\beta}^2d{\rm vol}_P + d\rho\wedge\theta + \rho d\theta - (d_P H)\circ u \wedge\beta  - H(u)d\beta.
\end{align*}
The energy density $\norm{du - X_H(u)\otimes\beta}^2dvol_P$ of the (vector-bundle-valued) 1-form $du - X_H(u)\otimes\beta$ can be estimated from below by the energy density of the projection of $du - X_H(u)\otimes\beta$ to $\{k\cdot X_Q(u) + \ell\cdot \partial_r\:|\:k, \ell\in\R\}$. Now, we compute the projection to $\{k\cdot X_Q(u)\:|\:k\in\R\}$:
\begin{align*}
\frac{\omega( du - X_H(u)\otimes\beta, J X_Q(u) )}{\norm{X_Q(u)}_J} &= \frac{\omega(J du - JX_H(u)\otimes\beta, - X_Q(u))}{\sqrt{\rho}}\\
							&= dQ(J du - JX_H(u)\otimes\beta)\cdot\rho^{-\frac{1}{2}}\\
							&= dQ( du\circ j - X_H\otimes\beta\circ j)\cdot\rho^{-\frac{1}{2}}\\
							&= \big(d(Q(u))\circ j - (d_PQ\circ j)(u)\\
							&\phantom{\&} - \{Q, H\}(u)\beta\circ j\big)\cdot\rho^{-\frac{1}{2}} = (d\rho\circ j - \rho\theta)\cdot\rho^{-\frac{1}{2}}.
\end{align*}
Similarly, the projection to $\{k\cdot \partial_r\:|\:k\in\R\}$ is equal to
\begin{align*}
    \frac{\omega( du - X_H(u)\otimes\beta, J \partial_r )}{\norm{\partial_r}_J} &=  \frac{\omega( du - X_H(u)\otimes\beta, J \rho\partial_r )}{\norm{\rho\partial_r}_J}\\
    &= \rho^{-\frac{1}{2}}\cdot \omega\left( du- X_H(u)\otimes\beta, -X_Q(u) \right)\\
    &= (d\rho + \rho\cdot \theta\circ j)\cdot \rho^{-\frac{1}{2}}.
\end{align*}
Using the considerations above and the orthogonality of $X_Q$ and $\partial_r$, we get
\begin{align*}
\norm{du - X_H(u)\otimes\beta}^2d{\rm vol}_P \geqslant& \left(\norm{(d\rho\circ j - \rho\theta)\cdot\rho^{-\frac{1}{2}}}^2 + \norm{(d\rho + \rho\cdot \theta\circ j)\cdot \rho^{-\frac{1}{2}}}^2\right)d{\rm vol}_P\\
            =& 2\cdot  \norm{(d\rho\circ j - \rho\theta)\cdot\rho^{-\frac{1}{2}}}^2 d{\rm vol}_P\\
            = & \frac{2}{\rho}\cdot (d\rho\circ j - \rho\theta)\circ j\wedge (d\rho\circ j - \rho\theta)\\
		=& \frac{2}{\rho}\cdot\left( -d\rho -\rho\theta\circ j \right)\wedge \left( d\rho\circ j -\rho\theta \right) \\
		=& \frac{2}{\rho}\cdot \big(-d\rho\wedge d\rho\circ j + \rho d\rho \wedge\theta\\
			&  - \rho (\theta\circ j)\wedge(d\rho\circ j) + \rho^2\cdot(\theta\circ j)\wedge\theta\big) \\
		=&\frac{2}{\rho}\cdot ( (d^C\rho)\wedge d\rho + 2 \rho d\rho\wedge\theta + \rho^2\cdot(\theta\circ j)\wedge\theta )\\
		\geqslant& \frac{2}{\rho}\cdot (d^C\rho)\wedge(d\rho) + 4 d\rho\wedge\theta.
\end{align*}
Therefore,
\begin{align*}
-dd^C\rho\geqslant& \frac{1}{\rho}\cdot d^C\rho\wedge d\rho + d\rho\wedge\theta - \rho d\theta + (d_PH)(u)\wedge\beta + H(u)d\beta.
\end{align*}
Since 
\[ d\mu=\frac{d\rho}{\rho},\quad dd^C\mu=\frac{dd^C\rho}{\rho} +\frac{d^C\rho\wedge d\rho}{\rho^2}, \]
we have
\begin{equation}\label{eq:importantineq}
-dd^C\mu + \theta\wedge d\mu + d\theta \geqslant \frac{(d_PH)(u)\wedge\beta + H(u)d\beta}{Q(u)}\geqslant 0. 
\end{equation}
This finishes the proof.
\end{proof}

The following lemma shows that there exists a real number $T>1$ such that the connected components of the preimage of $u^{-1}\left(\partial W\times [T, +\infty)\right)$ for every $u$ as in Theorem~\ref{thm:noescape} have uniformly bounded diameter.
\begin{lemma}
 Let $(H,J)$ be product data on $\hat{W}$ and let $E\in\R_{>0}$ be a positive real number. Then, there exist real numbers $T>1, L>0,$ and a compact subset $K\subset P$ such that the following holds. For every solution $u:P\to\hat{W}$ of the Floer equation $(du - X_H(u)\otimes\beta)^{0,1}=0$ with $\mathbb{E}(u)<E$ and every connected component $\Omega$ of $u^{-1}(\partial W\times [T,+\infty))$ at least  one of the following conditions holds:
 \begin{enumerate}
     \item $\Omega\subset K$,
     \item There exists an interval $I$ of length at most $L$ such that $\Omega\subset \iota_{c}(I\times\mathbb{S}^1)$ for some $c\in\{a_1, a_2, b\}.$
 \end{enumerate}
\end{lemma}
\begin{proof}
By Lemma~\ref{lem:gammainnbhd}, stated after this proof, there exist positive real numbers $\varepsilon_c, B_c\in\R_{>0},$ $c\in\{a_1,a_2, b\}$ such that $\gamma(\mathbb{S}^1)\subset \hat{W}\setminus (\partial W\times(B_c, +\infty))$ for every loop $\gamma:\mathbb{S}^1\to \hat{W}$ that satisfies
\[\int_0^1\norm{ \gamma'(t)- X_{H^c_t}(\gamma(t)) }_{J_c}^2 dt\leqslant \varepsilon_c.\]
Denote
\[T := \max_c B_c,\quad \varepsilon := \min_c \varepsilon_c,\quad L:= \frac{E}{\varepsilon} + 1. \] Let $K\subset P$ be the complement of 
\[\iota_{a_1}\left((-\infty, -L)\times \mathbb{S}^1\right)\:\cup\: \iota_{a_2}\left((-\infty, -L)\times \mathbb{S}^1\right)\:\cup\: \iota_{b}\left((L,+\infty)\times \mathbb{S}^1\right). \]
Now, we check that $T, L,$ and $K$ satisfy the conditions of the lemma. First, we show that for every connected component $\Omega'$ of $\Omega\cap \op{im} \iota_c$, where $c\in\{a_1, a_2, b\}$, there exists an interval $I$ of length at most $L$ such that $\Omega'\subset \iota_c\left( I\times\mathbb{S}^1 \right)$. Assume the contrary. Then, there exists an interval $[s_0, s_0+ L]$ such that $\iota_c \left( \{s\}\times\mathbb{S}^1 \right) \cap \Omega'\not=\varnothing$ for all $s\in[s_0, s_0+L]$. This implies
\[\int_0^1\abs{ \partial_t(u\circ\iota_c)(s,t)- X_{H^c_t}(u\circ\iota_c(s,t)) }_{J_c}^2 dt> \varepsilon_c\]
for all $s\in[s_0, s_0+L]$. Hence, if $c=b$,
\begin{align*}
    E & \geqslant \frac{1}{2} \int_P \abs{du - X_{H_p}\otimes\beta}^2 d{\rm vol}_P\\
    &\geqslant \frac{1}{2} \int_{\iota_c([s_0, s_0+L]\times\mathbb{S}^1)} \abs{du - X_{H_p}\otimes\beta}^2 d{\rm vol}_P\\
    & = \int_{s_0}^{s_0+L}\int_0^1 \abs{ \partial_t (u\circ\iota_c(s,t)) - X_{H^c_t}(u\circ\iota_c(s,t))}^2 dtds
\end{align*}
\begin{align*}
    & > \int_{s_0+1}^{s_0+L}\int_0^1 \abs{ \partial_t (u\circ\iota_c(s,t)) - X_{H^c_t}(u\circ\iota_c(s,t))}^2 dtds\\
    &> \int_{s_0+1}^{s_0+L} \varepsilon_c ds\\
    &= (L-1)\cdot \varepsilon_c\geqslant E,
\end{align*}
which is a contradiction. If $c\in\{a_1, a_2\}$, one similarly obtains a contradiction. If $\Omega\subset \bigcup_c \im{\iota_c}$, then the argument above implies that there exist $c\in\{a_1, a_2, b\}$ and an interval $I$ of length at most $L$ such that $\Omega\subset \iota_c(I\times\mathbb{S}^1)$. Otherwise, $\Omega\subset K$. Indeed, if $\Omega\not\subset K$ and if $\Omega\not \subset \bigcup_c\op{im}\iota_c$, then for some $c\in\{a_1, a_2, b\}$ and for some connected component $\Omega'$ of $\Omega\cap \op{im}\iota_c$ the sets $\iota_c \left( \{0\}\times\mathbb{S}^1 \right)\cap \Omega'$ and $\iota_c \left( \{(\op{sgn} c) \cdot L \}\times\mathbb{S}^1 \right)\cap \Omega'$ are non-empty (where $\op{sgn} c = -1$ if $c \in \{a_1, a_2\}$ and $\op{sgn} c = 1$ if $c = b$). This is a contradiction by previous considerations. 
\end{proof}

Now, we formulate the statement that was used in the proof of the last lemma and give a reference for the proof.
\begin{lemma}\label{lem:gammainnbhd}
    Let $(H, J)$ be a regular Floer data. Then, there exist $B, \varepsilon\in\R_{>0}$ such that for every smooth loop $\gamma: \mathbb{S}^1\to \hat{W}$ the following holds. If 
    \[\int_0^1\norm{ \gamma'(t)- X_{H_t}(\gamma(t)) }_J^2 dt\leqslant \varepsilon,\]
    then $\gamma(\mathbb{S}^1)\subset \hat{W}\setminus \left(\partial W\times(B, +\infty)\right).$
\end{lemma}
\begin{proof}
    See Lemma~3.6 in \cite{MU} or Proposition~2.2 in \cite{brocic2024bordism}.
\end{proof}

\begin{proof}[Proof of Theorem~\ref{thm:noescape}]
    Let $u:P\to\hat{W}$ denote a solution of the Floer equation $\left(du - X_H(u)\otimes \beta\right)^{0,1}=0$ that satisfies $\mathbb{E}(u)<E$. Let $T=T_{H,J}$ and $Q=Q^J$ be as in Proposition~\ref{prop:ineq}. Denote by  $\Omega$ a connected component of $u^{-1}(\partial W\times(T, +\infty)).$ By Proposition~\ref{prop:ineq}, the function $\mu:\Omega\to\R : p\mapsto \log Q(u(p), p)$ satisfies the inequality
    \[ -dd^C\mu + \theta\wedge d\mu + d\theta\geqslant 0, \]
    where
    \[\theta:=\frac{(d_PQ\circ j)(u) + \{Q, H\}(u)\beta\circ j}{Q\circ u}.\]
    In the view of Proposition~\ref{prop:aleksandrov} (the Alexandrov weak maximum principle), it is enough to show that $\norm{\theta}_{L^2}$ and $\norm{d\theta}_{L^2}$ are bounded by a constant that does not depend on $u$. Denote
    \[(w, r):= \left. u \right|_{\Omega} : \Omega\to \partial W\times (T, +\infty).\]
    Since $\partial W$ is compact and since
    \[\theta =\frac{(d_Pq\circ j)(w)}{q(w)} + \left(\frac{dq(X_h)}{q} + dh(R)\right)(w)\cdot \beta\circ j,\]
    it is enough to check that $\norm{dw}_{L^2}$ is bounded by a constant that does not depend on $u$. Here, $q$ and $h$ are the slopes of $Q$ and $H$, respectively. Let $g$ be a Riemannian metric on $\partial W$. Then, there exists $\varepsilon>0$ such that 
    \[\norm{\zeta + a r\partial_r}^2_J\geqslant \varepsilon r\cdot \norm{\zeta}_g^2\]
    for all $a\in \R$ and $\zeta\in T(\partial W)$. Hence,
    \begin{equation*}
        \norm{d w}_g^2 \leqslant \frac{1}{\varepsilon r}\cdot \norm{d u}_J^2\leqslant \frac{2}{\varepsilon r}\cdot\left( \norm{du - X_H(u)\otimes\beta}^2_J + \norm{X_H(u)\otimes\beta}^2_J \right).
    \end{equation*}
Since $\frac{1}{r}\norm{X_H(y,r)}_J^2$ is bounded and 
    \[\int_\Omega \norm{du - X_H(u)\otimes\beta}^2_J d{\rm vol}_P<E, \]
we finish the proof. 
\end{proof}

\medskip
\noindent {\bf Acknowledgement}. \sloppy The second author is supported by the Science Fund of the Republic of Serbia, grant no.~7749891, Graphical Languages - GWORDS. The third author is partially supported by National Key R\&D Program of China No.~2023YFA1010500, NSFC No.~12301081, NSFC No.~12361141812, and USTC Research Funds of the Double First-Class Initiative. Part of this paper has been presented in the Topology Seminar at the Yau Mathematical Sciences Center, Tsinghua University in May 2023, as well as in the conference Persistence Homology in Symplectic and Contact Topology, held in Albi, France in June 2023. We thank Honghao Gao and Jean Gutt for their invitations and hospitality. We also thank Prof. Yong-Geun Oh and Jungsoo Kang for their interest in our work and many fruitful communications. Part of this work (related to Theorem \ref{thm:bft2}) was done while the second author was visiting the Institute of Geometry and Physics at the University of Science and Technology of China in Hefei, and he is grateful for its hospitality. This part of work was also presented at the symplectic geometry seminar at the Institute for Advanced Study, Princeton, in March 2025 by the third author and he thanks Jean Gutt for his invitation. We express special gratitude to Dylan Cant for thorough and generous communications on several overlaps between our works. We thank Dylan Cant and Egor Shelukhin for pointing out a few inconsistencies in an earlier version of this paper that relate to the direction of the triangle inequality in Theorem \ref{thm:spectral-properties}. In addition, we thank Prof. Dietmar Salamon for a general discussion on the draft of this paper.

\bibliographystyle{amsplain0}
\bibliography{biblio_CHD}

\

\end{document}